\documentclass[10pt,a4]{article}
\usepackage{latexsym}
\usepackage[all]{xy}
\usepackage{amsmath, amsthm, amssymb, amsfonts}
\usepackage{enumerate}
\usepackage[latin1]{inputenc}
\usepackage{multirow}
\usepackage{hyperref}
\usepackage{abstract, indentfirst}
\usepackage{amscd, stmaryrd, setspace,color}
\usepackage{graphicx}

\newcommand{\ie}{{i.e. }}

\newcommand{\et}{\qquad\textrm{and}\qquad}

\newcommand{\tto}{\longrightarrow}
\newcommand{\mto}{\longmapsto}

\newcommand{\END}{\textrm{\sc End}}
\newcommand{\End}{\textrm{End}}

\newcommand{\un}{\mathbf{1}}

\newcommand{\CC}{\mathbb{C}}

\newcommand{\FF}{\mathbb{F}}
\newcommand{\GG}{\mathbb{G}}
\newcommand{\HH}{\mathbb{H}}

\newcommand{\NN}{\mathbb{N}}

\newcommand{\PP}{\mathbb{P}}

\newcommand{\cA}{\mathcal{A}}
\newcommand{\cB}{\mathcal{B}}
\newcommand{\cC}{\mathcal{C}}
\newcommand{\cD}{\mathcal{D}}

\newcommand{\cV}{\mathcal{V}}

\newcommand{\Set}{\mathsf{Set}}

\newcommand{\Vect}{\mathsf{Vect}}
\newcommand{\dgVect}{\mathsf{dgVect}}

\newcommand{\Coalg}{\mathsf{Coalg}}

\definecolor{mat}{rgb}{0.57,0.75,1}

\newtheorem{thm}{Theorem}[subsection]
\newtheorem{prop}[thm]{Proposition}
\newtheorem{cor}[thm]{Corollary}
\newtheorem{lemma}[thm]{Lemma}

\theoremstyle{definition}

\newtheorem{ex}[thm]{Example}

\newtheorem{rem}[thm]{Remark}

\newtheorem{hypo}[thm]{Hypothesis}

\begin{document}

\title{{\bf Cofree coalgebras over operads\\ and representative functions}
\vspace{1cm}}

\author{{\sc M. Anel}}

\maketitle

\begin{abstract}
We give a recursive formula to compute the cofree coalgebra $P^\vee(C)$ over any colored operad $P$ in $\cV=\Set,{\sf CGHaus},{\sf (dg)}\Vect$.
The construction is closed to that of \cite{Smith} but different. 
We use a more conceptual approach to simplify the proofs that $P^\vee$ is the cofree $P$-coalgebra functor and also the comonad generating $P$-coalgebras.

In a second part, when $\cV={\sf (dg)}\Vect$ is the category of vector spaces or chain complexes over a field, we generalize to operads the notion of representative functions of \cite{BL} and prove that $P^\vee(C)$ is simply the subobject of representative elements in the "completed $P$-algebra" $P^\wedge(C)$. This says that our recursion (as well as that of \cite{Smith}) stop at the first step.
\end{abstract}

\tableofcontents

\newpage

\section{Introduction}

This work has two parts.
In a first part we prove that coalgebras over a (colored) operad $P$ are coalgebras over a certain comonad $P^\vee$ by giving a recursive construction of $P^\vee$. In a second part we prove that the recursion is unnecessary in the case where the operad is enriched over vector spaces or chain complexes.

\paragraph{Cofree coalgebras and lax comonads}

The problem of constructing the cofree coalgebra $P^\vee(X)$ over a dg-operad $P$ was solved in \cite{Smith} following idea from \cite{Fox}.
We have prefered, however, a different construction to prove more easily that the coalgebras over the comonad $P^\vee$ are the $P$-coalgebras.
This result can be deduced from \cite{Smith} provided we know the category of $P$-coalgebras is comonadic but our approach does not use the comonadicity theorem. We deduce directly from the construction of $P^\vee$ that it is a comonad and that coalgebras over it coincides with $P$-coalgebras.
We have also try to work with (symmetric) operads in general symmetric monoidal category $\cV$ and not only chain complexes.

We recall that $P$-algebras can easily be seen to be algebras over a monad $P^\sim$ where $P^\sim$ is the analytic functor associated to $P$.
The situation is not so simple for $P$-coalgebras. The "completed $P$-algebra" functor $X\mapsto P^\wedge(X) = \prod_n[P(n),X^{\otimes n}]$, candidate to replace $P^\sim$, does not inherit a comonad structure from the operad structure of $P$ in general. (The obstruction for this is the non invertibility of the natural map $[A,B]\otimes [A',B']\to [A\otimes A',B\otimes B']$ in $\cV$.) However $P^\wedge$ is not far from begin a comonad.

The precise nature of $P^\wedge$ is to be a "lax comonad" by which we mean a lax monoidal functor from $\Delta_+^{op}$ to some endofunctor category. Details are given in section \ref{laxcomonad}. Essentially, a lax comonad is the data of functors $P_n^\wedge$ forming an augmented simplical diagram and natural maps $\alpha:P^\wedge_nP^\wedge_m \to P^\wedge_{n+m}$. If the maps $\alpha$ are isomorphisms, then the lax comonad is a genuine comonad.

The definition of a coalgebra over a comonad extends to a notion of a coalgebra over a lax comonad. Then, it is easy to see that $P$-coalgebras are exactly coalgebras over the lax comonad $P^\wedge$. Moreover, we show in proposition \ref{propcoref} that if $P^\wedge$ has a coreflection $Q$ into the subcategory of comonads, then $P^\wedge$-coalgebras coincides with $Q$-coalgebras. This proves not only that the cofree $P$-coalgebra can be described as a the cofree $Q$-coalgebra but also that the category of $P$-coalgebra is comonadic.
Our main result is then the following.
\begin{thm}[thm. \ref{mainthm1}]
The lax comonad $P^\wedge$ has a comonadic coreflection $P^\vee$.
\end{thm}

The proof of this theorem is the matter of section \ref {sectioncoreflection}. We construct $P^\vee$ by a recursion explained at the beginning of the section.
Although the idea of this recursion looks fairly general, we have not been able to prove it was valid without some mild assumptions on the symmetric monoidal category wher $P$ lives (see Hypothesis \ref{hypo}). Those assumptions are satisfied for categories such as sets, compactly generated Hausdorff topological spaces, toposes, in fact any cartesian closed category, but also vector spaces and chain complexes over a field.

\paragraph{Representative and recursive elements}

In the second part of this work, we prove that the recursion to construct $P^\vee$ stop at the first step when $P$ in a operad in vector spaces of chain complexes over a field. Since this first step is the same as in \cite{Smith}, this proves also that the recursion therein is unnecessary.
Our approach generalizes to operads methods for associative algebras; we start by recalling these.

\medskip
In \cite{Sweedler}, the author gives a construction of the cofree coassociative coalgebra as a special dual of the free algebra.
Precisely, the duality functor sending a coalgebra $C$ to the algebra $C^\star$ has a right adjoint $A\mapsto A^\circ$ and the $\circ$-dual of the free algebra $T(X)$ is the cofree coalgebra $T^\vee(X^\star)$ on the dual of $X$ (see \cite{Sweedler} or \cite{AJ} for a comprehensive treatment).
Then, in \cite{BL} and \cite{Hazewinkel}, the authors reduce the definition of $A^\circ$ to the fiber product
$$\xymatrix{
A^\circ \ar[rr]\ar[d] && A^\star \otimes A^\star\ar[d]\\
A^\star \ar[rr]^-{m^\star} && (A\otimes A)^\star
}$$
\ie the elements of $A^\circ$ are those of $A^\star$ whose diagonal decomposes in a {\sl finite} sum of tensor products.
Considering that $T(X)^\star$ is the completed tensor algebra $T^\wedge(X^\star)=\prod_n (X^\star)^{\otimes n}$, they prove that the cofree coalgebra $T^\vee(X)$ can be defined as the fiber product
$$\xymatrix{
T^\vee(X) \ar[rr]\ar[d] && T^\wedge(X) \otimes T^\wedge(X)\ar[d]\\
T^\wedge(X) \ar[rr]^-{\Delta = m^\star} && T^\wedge(X)\widehat \otimes T^\wedge(X)
}$$
where $T^\wedge(X)\widehat \otimes T^\wedge(X)$ is the completed tensor product and can be viewed as a subspace of $T^\wedge(X\oplus X)$.

$T^\vee(X)$ is a subspace of $T^\wedge(X)$ and the authors call the elements of $T^\wedge(X)$ belonging to $T^\vee(X)$ {\em representative elements}.
When $X$ is of finite dimension, we have $T^\wedge(X)= (T(X^\star))^\star$ and elements of $T^\wedge(X)$ can be thought as functions on $T(X^\star)$.
An element $f\in T^\wedge(X)$ is representative iff
there exists a finite number of elements $g_i, h_i\in T^\wedge(X)$ such that,
for any $a,b\in T(X^\star)$, 
$$
(\Delta f)(a\otimes b) = f(ab) = \sum_i g_i(a)\otimes h_i(b).
$$
The main lemma of \cite{BL} and \cite{Hazewinkel} is to prove that the $g_i$ and $h_i$ are still representative functions, \ie that 
$\Delta f\in T^\vee(X) \otimes T^\vee(X)\subset T^\wedge(X) \otimes T^\wedge(X)$.
To do this, they characterize representative elements as {\em recursive elements}, \ie elements $f$ such that the family of functions $a\mapsto f(ab)$, parametrized by $b\in T(X^\star)$, generate a finite dimensional vector space in $T^\wedge(X)$. Then it is easy to prove that if $f$ is recursive, $g_i$ and $h_i$ are recursive.

\bigskip
In order to do the same thing for coalgebras over an operad we need some adaptations.
The completed algebra $T^\wedge(X)$ can be replaced by the completed $P$-algebra $P^\wedge(X)=\prod_n[P(n),X^n]$. 
Because the associative operad was generated by a binary operation, the previous constructions involved only two terms tensor products of $T^\wedge(X)$.
But for a general operad, $T^\wedge(X) \otimes T^\wedge(X)$ and $T^\wedge(X)\widehat \otimes T^\wedge(X)$ have to be replaced by bigger objects 
$P^\wedge(P^\wedge(X))$ and $P^\wedge_2(X)$ taking into account all operations of $P$ (see section \ref{laxcomonad} for a definition of the latter).
Now our main result is the following, whose proof is the same as in \cite{BL} and \cite{Hazewinkel} characterizing representative elements as recursive ones.

\begin{thm}[thm. \ref{mainthm2}, cor. \ref{cormainthm2}]
The cofree $P$-coalgebra $P^\vee(X)$ can be defined as the fiber product
$$\xymatrix{
P^\vee(X) \ar[rr]\ar[d] && P^\wedge(P^\wedge(X))\ar[d]\\
P^\wedge(X) \ar[rr]^-{m^\star} && P^\wedge_2(X).
}$$
\end{thm}

Representative and recursive elements for operads are defined in sections \ref{represfct} and \ref{recursivefct}. 
The equivalence of the two notions is proven in proposition \ref{rep=rec}.

\paragraph{Einstein convention}

We have chosen to develop this work for colored operads. Notations for such objects can be quite heavy and we have found useful to introduce notational conventions inspired from Einstein summation convention in tensor calculus. These conventions are introduced in section \ref{Einstein} and shorten the notations for ends and coends of functors.

\paragraph{Acknowledgments}

This work was supported by the \href{http://www.math.ethz.ch/}{ETH} in Zürich through the Swiss National Science Foundation (project number $200021-137778$).

\section{The cofree coalgebra}

\subsection{Analytic functors}

We introduce analytic functors and colored operads following \cite{GJ}.

\paragraph{Free symmetric monoidal categories}

Let $X$ be a set, the free symmetric monoidal category on $X$ is noted $S(X)$. Its elements are sequences $(x_1, \dots, x_n)$ of element of $X$ of arbitrary but finite length. As in \cite{GJ}, we shall often use the notation $\overline{x}$ to refer to an arbitrary element of $S(X)$ when we do not need to reduce it to a sequence of elements of $X$.
The morphisms between two sequences $(x_1, \dots, x_n)$ and $(y_1, \dots, y_n)$ are defined to be the permutations $\sigma$ in the symmetric group $S_n$ such that $x_i=y_{\sigma(i)}$. 
There is no morphisms between two sequences $(x_1, \dots, x_n)$ and $(y_1, \dots, y_m)$ if they don't have the same length.
In particular, $S(X)$ is a groupoid and we have a canonical identification $S(X)^{op}= S(X)$.
We shall write $S^n(X)$ the subgroupoid of sequences of length $n$, we have $S(X)=\coprod_n S^n(X)$.

As constructed the category $S(X)$ is not a $\cV$-category but an ordinary category when $X$ is a set. 
We shall denote also by $S(X)$ the free $\cV$-category generated by $S(X)$, the hom objects are defined as the coproducts of the monoidal unit of $\cV$ indexed by the sets $S(X)(\overline{x}, \overline{y})$.

The construction $S(X)$ make sense when $X$ is a small category or $\cV$-category rather than a set. 
In this context $S(X)$ can be proven to be the free symmetric monoidal category on $X$, that is, if $\CC$ is a symmetric monoidal $\cV$-category, there is an equivalence of categories between the category of $\cV$-functors $X\to \CC$ and that of symmetric monoidal $\cV$-functors $S(X)\to \CC$ 
(see \cite{GJ} for details).

\paragraph{Free rigs}

If $\CC$ is a small symmetric monoidal $\cV$-category, the category $Pr(\CC)$ of $\cV$-preseheaves on $\CC$ can be equipped with the Day tensor product.
This product is defined as the cocontinous extension in each variable of the monoidal strucure of $\CC$.
Precisely, we define $\CC(-,x)\otimes_{Day}\CC(-,y)$ as $\CC(-,x\otimes y)$ and, if $F$ and $G$ are presheaves on $\CC$, using their description as colimit of representable presheaves, we define their tensor product by "$\cV$-linearity"
\begin{eqnarray*}
(F\otimes_{Day} G) &=& \left(\int^{x} F(x)\otimes \CC(-,x)\right)\otimes \left(\int^{y} G(y)\otimes \CC(-,y)\right)\\
&=& \int^{x,y} F(x)\otimes G(y)\otimes \CC(-,x)\otimes \CC(-,y)\\
&=& \int^{x,y} F(x)\otimes G(y)\otimes \CC(-,x\otimes y).
\end{eqnarray*}

Let $\cC$ be a $\cV$-category, tensored over $\cV$ presentable as an ordinary category and equipped with a symmetric monoidal structure enriched over $\cV$ (such categories are called rigs in \cite{GJ}).
There exists an equivalence of categories between
$\cV$-cocontinuous symmetric monoidal $\cV$-functors $Pr(\CC)\to \cC$ (rig morphism) and
symmetric monoidal $\cV$-functors $\CC\to \cC$.
This makes $Pr(\CC)$ into the free rig generated by $\CC$.

In the case where $\CC=S(X)$, the rig $Pr(S(X))$ is the free rig on $X$. 
Composing the previous equivalences, there exists an equivalence of categories between
$\cV$-cocontinuous symmetric monoidal $\cV$-functors $Pr(S(X))\to \cC$ and
$\cV$-functors $X\to \cC$.

\paragraph{$S$-distributors and analytic functors}

Let $X$ and $Y$ be two sets, a rig morphism $F:Pr(S(Y))\to Pr(S(X))$ is equivalent to the data of a functor $F:Y\to Pr(S(X))$ or equivalently to a $(S(X),Y)$-bimodule $F:S(X)^{op}\times Y\to \cV$. Such a bimodule is called an {\em $S$-distributor} in \cite{GJ}.

Forgetting the monoidal structure, the universal property of presheaves says that a cocontinuous functor $F:Pr(S(Y))\to Pr(S(X))$ is equivalent to a functor $S(Y)\to Pr(S(X))$, that is to a $(S(Y),S(X))$-bimodule $\FF:S(X)^{op}\times S(Y)\to \cV$. We shall call such a bimodule an {\em $S$-bimodule between $X$ and $Y$}. The $S$-bimodule $\FF$ corresponding to an $S$-distributor $F$ is given by
$$
\FF(\overline{x}; \overline{y}) = \int^{\overline{z_1},\dots, \overline{z_n}\in S(X) }S(X)(\overline{x};\overline{z_1}\dots\overline{z_n})\otimes \bigotimes_i F(\overline{z_i}; y_i) 
$$
Equivalently $\FF(-; \overline{y})$ can be caracterized as the iterated Day tensor product 
$F(-;y_1)\otimes_{Day}\dots \otimes_{Day}F(-;y_n)$.

The composition of two rig morphisms $F:Pr(S(Y))\to Pr(S(X))$ and $G^\sim:Pr(S(Z))\to Pr(S(Y))$  give the following composition formula for the corresponding $S$-distributors $F$ and $G$:
$$
(F\circ G)(\overline{x};z) = \int^{\overline{y}\in S(Y)} F(\overline{y};z) \GG(\overline{x};\overline{y}).
$$

To a morphism of free rigs $F:Pr(S(Y))\to Pr(S(X))$ is associated an {\em analytic functor} $F^\sim:\cV^X \to \cV^Y$ given by the formula
$$
F^\sim(V)(y) = \int^{\overline{x}\in S(X)} F(\overline{x};y) \otimes V(x_1)\otimes \dots \otimes V(x_n).
$$
This correspondance is natural in $F$ and compatible with composition, that is we have $(F\circ G)^\sim = F^\sim \circ G^\sim$.
This correspondance is also fully faithful.

\subsection{Einstein convention}\label{Einstein}

In order to have more compact notation for ends and coends, we are going to make the following conventions inspired from tensor calculus.
\begin{itemize}

\item Contravariant variables of functors will be noted in an upper position and covariant variables in an lower position: a writting such as $F^{a,b}_c$ refers to the values (we shall say the {\em components}) $F(a,b,c)$ of a functor $F:A^{op}\times B^{op}\times C\to \cV$. The categories $A$, $B$, $C$ in question should be clear from the context. They could be ordinary categories of $\cV$-categories, in which case the functor $F$ is assumed to be a $\cV$-functor.

\item We shall use the same notation for natural transformation and write $\alpha^a_b:F^a_b\to G^a_b$ the component of a natural transformation $\alpha:F\to G$.

\item Tensor products in $\cV$ shall sometimes be written with concatenation. For example the elements $F^aG_b$ are the components of the functor $F\otimes G:A^{op}\times B\to \cV$ with values $F(a)\otimes G(b)$.
The internal hom of $\cV$ shall be noted $[-,-]$, a formula such as $[F_a,G_b]$ refers to the components of the functor $[F,G]:A^{op}\times B\to \cV$ with values $[F(a), G(b)]$. Notice the change of variance of $a$, if $H=[F,G]$, then $H^a = [F_a,G_b]$.

\item In a notation like $F^{a,b}G_b$, when a letter is repeated twice, once in a lower and once in an upper position, we will assume that an implicit coend is taken over this variable:
$$
F^{a,b}G_b = \int^{b\in B}F(a,b)\otimes G(b).
$$
Similarly, in notations like $[F^{a,b},G^b]$ or $[F^a_b,G_b]$, when a letter appears twice, once on each variable of a hom $[-,-]$, both time in an upper position or in a lower position, we will assume that an implicit end is taken over this variable:
$$
[F^{a,b},G^b] = \int_{b\in B}[F(a,b), G(b)].
$$
The categories over which to take the end or coend should be clear from the context.

\item Notations such as $F^{a,a}$, $[F^a,G_a]$, $F^{a,b}G_b H_b$ or $[F^aG^a,H_a]$ are forbidden. But $[F^a,[G_a^b,H^b]]$ is correct.

\end{itemize}

We illustrate these conventions with some classical formulas concerning Yoneda's lemma. We shall use these formula in computations later.
Let $\CC$ be a small $\cV$-category, and $Pr(\CC)$ be the category of $\cV$-presheaves, \ie $\cV$-functor $\CC^{op}\to \cV$.
The fully-faithfullness of Yoneda's embedding $\CC\to Pr(\CC)$ can be written as the following end formula, for any $G:\CC^{op}\to \cV$
$$
\int_{c\in \CC} [\CC(c,d),G(c)] = G(d) \iff [\CC^c_d,G^c] = G^d.
$$
Any presheaf $G$ is also a weighted colimit of representable presheaves, this is best written as a coend
$$
\int^{c\in \CC} F(c)\otimes \CC(d,c) = F(d) \iff F^c\CC^d_c = F^d
$$
which is another instance of Yoneda's lemma.


\medskip
We recall the fundamental adjunctions for ends and coend. For three functors $F^a_b$, $G^b_c$ and $H^a_c$, we have bijections between the following sets:
\begin{center}
\begin{tabular}{lc}
\rule[-2ex]{0pt}{4ex} transformations natural in $a,b$ & $F^a_bG^b_c\to  H^a_c$, \\
\rule[-2ex]{0pt}{4ex} transformations natural in $a,c$ & $F^a_b\to [G^b_c,H^a_c]$, \\
\rule[-2ex]{0pt}{4ex} transformations natural in $b,c$ & $G^b_c\to [F^a_b,H^a_c]$.
\end{tabular}
\end{center}
For example, when all variable $a,b,c$ belong to the same category $C$, the monoidal structure on $C$-bimodules induced by the coend is the matrix product and the end compute its right adjoint in each variable.

\bigskip
When some of the categories of variables are of the type $S(X)$, we introduce more conventions for functors, ends and coends.
\begin{itemize}

\item Components of a functor $F:Y\times S(X)^{op}\to \cV$ shall be written as $F^{\overline{x}}_y$ and
components of a functor $F:S(Y)\times S(X)^{op}\to \cV$ shall be written as $F^{\overline{x}}_{\overline{y}}$.

For example the $S$-bimodule $\FF$ associated to and $S$-distributor $F$ as components
$$
\FF^{\overline{x}}_{\overline{y}} = S(X)^{\overline{x}}_{\overline{z_1}\dots\overline{z_n}}  F^{\overline{z_1}}_{y_1}\dots  F^{\overline{z_n}}_{y_n}
$$

\item We shall use comma to separate the different variables of same variance of a functor to contrast with the use of concatenation used for the product of elements in $S(X)$. This way $F^{x_1,\dots,x_n}$ or $G_{\overline{x_1},\dots,\overline{x_n}}$ are the values of functors of $n$ variables but $F^{x_1\dots x_n}$ and $G_{\overline{x_1}\dots\overline{x_n}}$ are the values of functors of a single variable in $S(X)$.

\item In a formula like $F^{\overline{x}}_yG^{z_1}_{x_1}\dots G^{z_n}_{x_n}$ where the same letter $x$ is repeated in upper and lower position
once with an overline and all others instances with indices, we shall assume that $\overline{x}=x_1\dots x_n$ and take an implicit coend over $\overline{x}\in S(X)$:
$$
F^{\overline{x}}_yG^{z_1}_{x_1}\dots G^{z_n}_{x_n} = \coprod_\NN \int^{(x_1,\dots x_n)\in S^n(X)}
F(x_1\dots x_n,y)\otimes G(z_1,x_1)\otimes \dots\otimes  G(z_n,x_n).
$$
The variables $z_i$ and their dependance on the size of $\overline{x}$ should be clear from the context.

For such a coend to make sense, the part of the formula where $x$ appears with indices need to be a functor on $S(X)$.
This functorality should always be clear from the context. In the example, the morphisms of $S(X)$ act on the $G$s by permutations.

For example the composition of $S$-distributors is given in components by
$$
(F\circ G)^{\overline{x}}_{z} = F^{\overline{y}}_z \GG^{\overline{x}}_{\overline{y}} = 
F^{\overline{y}}_z 
S(X)^{\overline{x}}_{\overline{z_1}\dots\overline{z_n}}  G^{\overline{z_1}}_{y_1}\dots G^{\overline{z_n}}_{y_n}.
$$

\end{itemize}

\subsection{Operads and associated functors}\label{piepsilon}

Let $X$ be a set, an {\em $X$-colored symmetric $\cV$-operad $P$} (we shall say simply an {\em operad}) is the data of
\begin{itemize}
\item an $S$-distributor $P$ with colors $X$, \ie a functor $P:S(X)^{op}\times X\to \cV$
\item a monoid structure on $P$ for the composition of $S$-distributors.
\end{itemize}

The components of $P$ shall be noted $P^{\overline{x}}_y$ or $P(\overline{x};y)$. 
In case the set of colors as one element $P^{\overline{x}}_y$ depends only on the length of $\overline{x}$ and is simply noted $P(n)$.

\medskip
In components, 
the unit is given by maps $Id^{\overline{x}}_y\to P^{\overline{x}}_y$ where $Id^{\overline{x}}_y$ are the components of the identity operad ($Id^{\overline{x}}_y = 0$, the initial object of $\cV$, unless $\overline{x}=y$, in which case $Id^{\overline{x}}_y = \un$, the monoidal unit of $\cV$)
and the composition is given by maps
$$
P^{\overline{z}}_y P^{\overline{x_1}}_{z_1}\dots P^{\overline{x_n}}_{z_n} = P^{\overline{x}}_y\Big(\bigotimes_iP^{\overline{z_i}}_{x_i}\Big) \tto P^{\overline{x_1}\dots \overline{x_n}}_y.
$$
The associativity condition is equivalent to the commutation of the squares
$$\xymatrix{
P^{\overline{x}}_y
\left(\bigotimes_i P^{\overline{z_i}}_{x_i}
\left(\bigotimes_{j_i}P^{\overline{t_{i,j_i}}}_{z_{i,j_i}}\right)\right)
\ar[d]\ar[rr]&&
P^{\overline{x}}_y
\left(\bigotimes_i P^{\overline{t_{i,1}}\dots \overline{t_{i,m_i}}}_{x_i}\right)
\ar[d]\\
P^{\overline{z_1}\dots \overline{z_n}}_y
\left(\bigotimes_{i,j_i}P^{\overline{t_{i,j_i}}}_{z_{i,j_i}}\right)\
\ar[rr]&&
P^{\overline{t_{1,1}}\dots \overline{t_{n,m_n}}}_y
}$$
We leave the reader to explicit in components the identity conditions.

\bigskip
We associate to $P$ several objects.
\begin{itemize}
\item The {\em analytic functor of $P$} is the functor $P^\sim:\cV^X\to \cV^X$ defined by 
$$
P^\sim(A)_y = P^{\overline{x}}_yA_{\otimes \overline{x}} = P^{\overline{x}}_yA_{x_1}\dots A_{x_n}.
$$
$P^\sim$ is the free $P$-algebra functor.
Because $P$ is an operad, $P^\sim$ is a monad.

\item The {\em co-analytic functor of $P$} is the functor $P^\wedge:\cV^X\to \cV^X$ defined by 
$$
P^\wedge(C)^y = [P^{\overline{x}}_y,C^{\otimes \overline{x}}] = [P^{\overline{x}}_y,C^{x_1}\dots C^{x_n}].
$$

\begin{rem}
The notation for variance makes it clear that, if $X$ was a category, this functor would be defined on $\cV^{X^{op}}$.
\end{rem}

\begin{rem}
This functor could be called the "completed $P$-algebra functor" but we shall not use this name as the notation of variance makes it clear that there is no natural map $P^\sim(A)\to P^\wedge(A)$. (This is even clearer when $X$ is replaced by a category where $P^\sim$ and $P^\wedge$ are not defined on the same categories.) However there is a natural map $P^\sim(A)^\star\to P^\wedge(A^\star)$ (where $B^\star=[B,\un]$ is the dual of $A$ in $\cV$). 

Nonetheless, it is this object $P^\wedge(C)$ that will play the role of $T^\wedge(X)$ in the construction of the cofree $P$-coalgebra.
\end{rem}

\begin{rem}
$P$ being an operad does not imply that $P^\wedge$ is a comonad. 
We give some detail as this is somehow the source of all the trouble to construct the cofree $P$-coalgebra.
Let us consider the following diagram of natural maps
$$\xymatrix{
& [P^{\overline{z}}_y,\otimes_i [P^{\overline{x_i}}_{z_i},C^{\otimes \overline{x_i}}]] \ar[d]\\
[P^{\overline{x}}_y,C^{\otimes \overline{x}}]\ar[r]\ar@{..>}[ru] &[P^{\overline{z}}_y, [\otimes_iP^{\overline{x_i}}_{z_i},\otimes_iC^{\otimes \overline{x_i}}]]
}$$
where the bottom map is induced by the multiplication of $P$ and the vertical map is induced by the monoidal structure.
We have $P^\wedge(P^\wedge(C))^y = [P^{\overline{z}}_y,\otimes_i [P^{\overline{x_i}}_{z_i},C^{\otimes \overline{x_i}}]]$
and a map $P^\wedge(C)^y\to P^\wedge(P^\wedge(C))^y$ would be a diagonal lift in the previous diagram.
But there can be such a lift, natural in $C$, only if the vertical is an isomorphism, which is almost never the case.

\end{rem}

\item The PROP of $P$ is the symmetric monoidal $S(X)$-category $\PP$ defined by
$$
\PP^{\overline{x}}_{\overline{y}} = 
P^{\overline{z_1}}_{y_1}\dots P^{\overline{z_n}}_{y_n} 
S(X)^{\overline{x}}_{\overline{z_1}\dots \overline{z_n}}.
$$
This definition of $\PP$ is actually the Day product $(P^{(-)}_{y_1}\otimes_{Day}\dots \otimes_{Day}P^{(-)}_{y_n})^{\overline{x}}$.

The composition $\PP^{\overline{y}}_{\overline{z}}\PP^{\overline{x}}_{\overline{y}} \tto \PP^{\overline{x}}_{\overline{z}}$ is 
implied by the monoid structure of $P$ and essentially given by 
$$
P^{\overline{y_1}}_{z_1}\dots P^{\overline{y_n}}_{z_n} 
\left(P^{\overline{x_{11}}}_{y_{11}}\dots P^{\overline{x_{1m_1}}}_{y_{1m_1}}\right)
\dots 
\left(P^{\overline{x_{n1}}}_{y_{n1}} \dots P^{\overline{x_{nm_n}}}_{y_{nm_n}}\right)
\tto
P^{\overline{x_{11}}\dots \overline{x_{1m_1}}}_{z_1}\dots P^{\overline{x_{n1}}\dots \overline{x_{nm_n}}}_{z_n}.
$$

The symmetric monoidal structure is given by concatenation on the objects and on the arrows the structure maps
$\PP^{\overline{x}}_{\overline{y}}\PP^{\overline{x'}}_{\overline{y'}} \tto \PP^{\overline{x}\overline{x'}}_{\overline{y}\overline{y'}}$
are induced by the natural maps
$P^{\overline{x_1}}_{y_1}\dots P^{\overline{x_n}}_{y_n} \tto \PP^{\overline{x_1}\dots \overline{x_n}}_{y_1\dots y_n}$
existing by definition of $\PP$ as a coend.

\begin{rem}
The composition of the operad $P$ can be rewritten using $\PP$ as a map $P_y^{\overline{x}}\PP^{\overline{z}}_{\overline{x}} \tto P_y^{\overline{z}}$. It is actually a particular case of composition in $\PP$ since $P_y^{\overline{x}}=\PP_y^{\overline{x}}$.
\end{rem}

\item To $\PP$ is associated a cocontinous functor $\PP:\cV^{S(X)}\to \cV^{S(X)}$ given by 
$$
\PP(M)_{\overline{y}} = \PP^{\overline{x}}_{\overline{y}} M_{\overline{x}}
$$
and a continous functor $\PP^\wedge :\cV^{S(X)^{op}}\to \cV^{S(X)^{op}}$ given by 
$$
\PP^\wedge(N)^{\overline{y}} = [\PP^{\overline{x}}_{\overline{y}},N^{\overline{x}}].
$$
We emphasize the difference of variance on $S(X)$ for the two functors. In particular, these two functors are not adjoint to each other (although they do have adjoints but we shall not consider them).

\begin{rem}
Because $\PP$ is a category, the functor $\PP$ is a monad and the functor $\PP^\wedge$ is a comonad.
This is actually the point of considering $\PP^\wedge$, this comonadic functor will be a sort of replacement for the non-existent comonadic structure on $P^\wedge$. 
\end{rem}

\begin{rem}
The monoidal structure for $\PP$ is responsible for the functor $\PP$ to be monoidal if $\cV^{S(X)} = Pr(S(X)^{op})$ is endowed with the Day product, but we shall not use this. The functor $\PP^\wedge$ is only lax monoidal.
\end{rem}
\end{itemize}

We shall also need four functors $\epsilon:\cV^X\to \cV^{S(X)}$, $\epsilon:\cV^X\to \cV^{S(X)^{op}}$, $\pi:\cV^{S(X)}\to \cV^X$ and $\pi:\cV^{S(X)^{op}}\to \cV^X$.
The $\epsilon$ are defined respectively by 
$$
\epsilon(A)_{\overline{x}} = A_{\otimes \overline{x}} = A_{x_1}\dots A_{x_n} \et 
\epsilon(C)^{\overline{x}} = C^{\otimes \overline{x}} = C^{x_1}\dots C^{x_n}
$$
with morphisms permuting the factors.
The projections $\pi$ are defined respectively by 
$$
\pi(A)_{x} = A_x \et 
\pi(C)^{x} = C^x.
$$

There is a number of formulas relating $\epsilon$ and $\pi$ with the functors associated to $P$
	\begin{itemize}
	\item $\pi\epsilon=id$
	\item $\pi\PP\epsilon = P^\sim$
	\item $\pi\PP^\wedge\epsilon = P^\wedge$
	\item $\epsilon P^\sim = \PP\epsilon$ (and thus $\pi\PP^n\epsilon = (P^\sim)^n$)
	\item and there are maps $\alpha_{n,m}:\pi(\PP^\wedge)^n\epsilon\pi(\PP^\wedge)^m\epsilon\to \pi(\PP^\wedge)^{n+m}\epsilon$.
	\end{itemize}
The first relations are obvious. 
We prove that $\epsilon P^\sim = \PP\epsilon$:
\begin{eqnarray*}
(\epsilon P^\sim)(A)_{\overline{y}} &=& \left(P^{\overline{x_1}}_{y_1}A_{\otimes \overline{x_1}}\right)\dots \left(P^{\overline{x_n}}_{y_n}A_{\otimes \overline{x_n}}\right)\\
&=& \left(P^{\overline{x_1}}_{y_1}\dots P^{\overline{x_n}}_{y_n}\right)A_{\otimes \overline{x_1}\dots \overline{x_n}}\\
&=& \left(P^{\overline{x_1}}_{y_1}\dots P^{\overline{x_n}}_{y_n}\right)S(X)_{\overline{x_1}\dots \overline{x_n}}^{\overline{z}}A_{\otimes \overline{z}}\\
&= & \PP^{\overline{z}}_{y_1\dots y_n} A_{\otimes \overline{z}}= \PP(\epsilon A)_{\overline{y}}.
\end{eqnarray*}
We have used Yoneda lemma to compute $A_{\otimes \overline{x_1}\dots \overline{x_n}} = S(X)_{\overline{x_1}\dots \overline{x_n}}^{\overline{z}}A_{\otimes \overline{z}}$.


\medskip

The maps $\pi(\PP^\wedge)^n\epsilon\pi(\PP^\wedge)^m\epsilon\to \pi(\PP^\wedge)^{n+m}\epsilon$ are constructed by iterating the canonical map $\otimes_i[A_i,B_i]\to [\otimes_i A_i,\otimes_i B_i]$.
\begin{eqnarray*}
(\pi(\PP^\wedge)^n\epsilon\pi(\PP^\wedge)^m\epsilon)(C)^{y} &=& 
[P^{\overline{x_1}}_{y},\dots [\PP^{\overline{x_n}}_{\overline{x_{n-1}}}, \otimes_i [P^{\overline{z_{1,i}}}_{x_{n,i}},\dots [\PP^{\overline{z_{m,i}}}_{\overline{z_{m-1,i}}},C^{\otimes \overline{z_{m,i}}}]\dots ]\\
&\to & [P^{\overline{x_1}}_{y},\dots [\PP^{\overline{x_{n}}}_{\overline{x_{n-1}}}, 
[\otimes_i P^{\overline{z_{1,i}}}_{x_{n,i}},\otimes_i [\PP^{\overline{z_{2,i}}}_{\overline{z_{1,i}}}, \dots [\PP^{\overline{z_{m,i}}}_{\overline{z_{m-1,i}}},C^{\otimes \overline{z_{m,i}}}]\dots ]\\
&\dots&\\
&\to & [P^{\overline{x_1}}_{y},\dots [\PP^{\overline{x_{n}}}_{\overline{x_{n-1}}}, 
[\otimes_i P^{\overline{z_{1,i}}}_{x_{n+1,i}},\dots
[\otimes_i\PP^{\overline{z_{m,i}}}_{\overline{z_{m-1,i}}},\otimes_iC^{\otimes \overline{z_{m,i}}}]\dots ]
\end{eqnarray*}
Now we have to prove that the last term is $\pi(\PP^\wedge)^{n+m}\epsilon(C)^y$.
By Yoneda lemma, we have 
$$
\otimes_{i=1}^NC^{\otimes \overline{z_{m,i}}} = C^{\otimes \overline{z_{m,1}}\dots \overline{z_{m,N}}} = [S(X)_{\overline{z_{m,1}}\dots \overline{z_{m,N}}}^{\overline{u_m}},C^{\otimes \overline{u_m}}].
$$
Using this and the definition of $\PP$, we get
\begin{eqnarray*}
[\otimes_i\PP^{\overline{z_{m,i}}}_{\overline{z_{m-1,i}}},\otimes_iC^{\otimes \overline{z_{m,i}}}] 
&=& [\otimes_i\PP^{\overline{z_{m,i}}}_{\overline{z_{m-1,i}}},[S(X)_{\overline{z_{m,1}}\dots \overline{z_{m,N}}}^{\overline{u_m}},C^{\otimes \overline{u_m}}]]\\
&=& [\otimes_i\PP^{\overline{z_{m,i}}}_{\overline{z_{m-1,i}}}S(X)_{\overline{z_{m,1}}\dots \overline{z_{m,N}}}^{\overline{u_m}},C^{\otimes \overline{u_m}}]\\
&=& [\PP^{\overline{u_m}}_{\overline{z_{m-1,1}}\dots \overline{z_{m-1,N}}},C^{\otimes \overline{u_m}}].
\end{eqnarray*}
Continuing, we get
\begin{eqnarray*}
[\otimes_i\PP^{\overline{z_{m-1,i}}}_{\overline{z_{m-2,i}}},[\PP^{\overline{u_m}}_{\overline{z_{m-1,1}}\dots \overline{z_{m-1,N}}},C^{\otimes \overline{u_m}}]]
&=& [\otimes_i\PP^{\overline{z_{m-1,i}}}_{\overline{z_{m-2,i}}},[SX_{\overline{z_{m-1,1}}\dots \overline{z_{m-1,N}}}^{\overline{u_{m-1}}},[\PP^{\overline{u_m}}_{\overline{u_{m-1}}},C^{\otimes \overline{u_m}}]]]\\
&=& [\otimes_i\PP^{\overline{z_{m-1,i}}}_{\overline{z_{m-2,i}}}SX_{\overline{z_{m-1,1}}\dots \overline{z_{m-1,N}}}^{\overline{u_{m-1}}},[\PP^{\overline{u_m}}_{\overline{u_{m-1}}},C^{\otimes \overline{u_m}}]]\\
&=& [\PP_{\overline{z_{m-2,1}}\dots \overline{z_{m-2,N}}}^{\overline{u_{m-1}}},[\PP^{\overline{u_m}}_{\overline{u_{m-1}}},C^{\otimes \overline{u_m}}]].
\end{eqnarray*}
Iterating this, we finally get
\begin{eqnarray*}
[P^{\overline{x_1}}_{y},\dots [\PP^{\overline{x_{n+1}}}_{\overline{x_n}}, 
[\otimes_i \PP^{\overline{z_{1,i}}}_{x_{n+1,i}},\dots
[\otimes_i\PP^{\overline{z_{m+1,i}}}_{\overline{z_{m,i}}},\otimes_iC^{\otimes \overline{z_{m+1,i}}}]\dots ] &=& 
[P^{\overline{x_1}}_{y},\dots [\PP^{\overline{x_{n+1}}}_{\overline{x_n}}, 
[\PP^{\overline{u_{1}}}_{\overline{x_{n+1}}},\dots 
[\PP^{\overline{u_m}}_{\overline{u_{m-1}}},C^{\otimes \overline{u_m}}]\dots ] \\
&=& \pi(\PP^\wedge)^{n+m}\epsilon(C)^y.
\end{eqnarray*}

\subsection{Coalgebras over an operad}\label{coalgebras}

We recall first the notion of algebra for matter of comparison.
An {\em algebra over an operad} $P$ colored by $X$ is an object $A$ in $\cV^X$ together with maps
$$
P^{\overline{x}}_{y}A_{\otimes \overline{x}}= P^{\overline{x}}_{y}A_{x_1}\dots A_{x_n}\tto A_y
$$
satisfying the unit condition: the composite
$A_y = Id^{\overline{x}}_yA_{\otimes \overline{x}} \to P^{\overline{x}}_yA_{\otimes \overline{x}} \to A_y$ 
must be the identity of $A_y$;
and the associativity condition: the square 
$$
\xymatrix{
P^{\overline{x}}_y
\left(\bigotimes_i P^{\overline{z_i}}_{x_i}
A_{\otimes \overline{z_i}}\right)
\ar[d]\ar[rr]&&
P^{\overline{x}}_y
A_{\otimes \overline{x}}
\ar[d]\\
P^{\overline{z_1}\dots \overline{z_n}}_y
A_{\otimes \overline{z_1}\dots \overline{z_n}}
\ar[rr]&&
A_y
}$$
must commute.

The structure maps of a $P$-algebra $A$ reduce to a single map $P^\sim(A)\to A$ in $\cV^X$ and the algebra condition is precisely the condition for this map to endow $A$ with the strucure of an algebra over the monad $P^\sim$:
the unit condition says that the composition $A\to P^\sim(A)\to A$ should be the identity of $A$
and the associativity condition is the commutation of
$$\xymatrix{
P^\sim (P^\sim (A)) \ar[r]\ar[d]& P^\sim (A)\ar[d]\\
P^\sim (A) \ar[r] &A.
}$$
The situation is not as nice for coalgebras.

\bigskip

A {\em coalgebra over an operad} $P$ colored by $X$ is an object $C$ in $\cV^X$ together with a maps
$$
C^yP^{\overline{x}}_{y} \tto C^{\otimes \overline{x}}= C^{x_1}\dots C^{x_n} \iff \delta_y:C^y\tto [P^{\overline{x}}_{y}, C^{\otimes \overline{x}}]
$$
satisfying the counit condition: the composite
$C^y\tto [P^{\overline{x}}_{y}, C^{\otimes \overline{x}}]\to [Id^{\overline{x}}_{y}, C^{\otimes \overline{x}}]=C^y$ 
must be the identity of $A_y$;
and the coassociativity condition: the square 
$$
\xymatrix{
C^yP^{\overline{x}}_y\left(\bigotimes_i P^{\overline{z_i}}_{x_i}\right) 
\ar[d]\ar[rr]&&
\bigotimes_i C^{x_i}P^{\overline{z_i}}_{x_i} 
\ar[d]\\
C^yP^{\overline{z_1}\dots \overline{z_n}}_y
\ar[rr]&&
C^{\otimes \overline{z_1}\dots \overline{z_n}}
}$$
must commute.

The structure maps of a coalgebra can be written as a single map $\delta:C\to P^\wedge(C)$. $P^\wedge$ has a natural projection $P^\wedge(C)\to C$ and the counit condition is equivalent to the condition that $C\to P^\wedge(C)\to C$ is the identity of $C$. 
However, $P^\wedge$ is not a comonad and the coassociativity condition cannot be expressed with $P^\wedge$ alone. 
The best we can do to rewrite the coassociativity is the commutativity of the diagram
$$\xymatrix{
C^y\ar[r]\ar[d]&[P^{\overline{x}}_{y}, C^{\otimes \overline{x}}]\ar[r]&[P^{\overline{x}}_{y}, \bigotimes_i [P^{\overline{z_i}}_{x_i}, C^{\otimes \overline{z_i}}]]\ar[d]\\
[P^{\overline{z}}_{y}, C^{\otimes \overline{z}}]\ar[r]&[P^{\overline{x}}_{y}\PP^{\overline{z}}_{\overline{x}}, C^{\otimes \overline{x}}]\ar@{=}[r]& [P^{\overline{x}}_{y}, [\PP^{\overline{z}}_{\overline{x}}, C^{\otimes \overline{z}}]]
}$$
which can also be written as
$$\xymatrix{
C\ar[d]_{\delta}\ar[r]^-{\delta}&P^\wedge(C)\ar[rr]^-{P^\wedge(\delta)}&&P^\wedge(P^\wedge(C))\ar[d]^{\alpha_{1,1}}\\
P^\wedge(C)\ar[rrr]^{\pi \Delta\epsilon}&&&(\pi\PP^\wedge\PP^\wedge\epsilon)(C).
}$$
where we have used the map $\alpha_{1,1}:(P^\wedge)^2\to \pi(\PP^\wedge)^2\epsilon$ and the comonad structure $\Delta:\PP^\wedge\to \PP^\wedge\PP^\wedge$ defined in the previous section.
We shall see in the section \ref{laxcomonad} that this structure on $C$ is that of a coalgebra over a lax comonad.

\medskip
The situation is nicer for maps and they can be characterized using only $P^\wedge$.
A {\em map of $P$-coalgebras} is a map $f:C\to D\in \cV^X$ such that the equivalent diagrams
$$
\vcenter{\xymatrix{
C^y\ar[r]\ar[d]_{f^y} & [P^{\overline{x}}_{y}, C^{\otimes \overline{x}}]\ar[d]^{[P^{\overline{x}}_{y}, f^{\overline{x}}]}\\
D^y\ar[r] &  [P^{\overline{x}}_{y}, C^{\otimes \overline{x}}]
}}
\quad\iff\quad
\vcenter{\xymatrix{
C\ar[r]\ar[d]_f & P^\wedge(C)\ar[d]^{P^\wedge(f)}\\
D\ar[r] &  P^\wedge(D)
}}
$$
commutes.

\subsection{The comonad of coendomorphisms}\label{coendomorphisms}

We recall a few properties of Kan extensions. Considering a diagram
$$\xymatrix{
\cA\ar[d]_H\ar[rr]^-F&& \cB\\
\cC\ar[rru]_-G
}$$
where $\cB$ and $\cC$ are bicomplete $\cV$-categories, the composition of functors admits both a left and a right adjoint.
The left Kan extension of $F$ along $H$ is the functor $H_!F$ such that there is a bijection between
\begin{center}
\begin{tabular}{lc}
\rule[-2ex]{0pt}{4ex} natural transformations & $F\to G\circ H$, \\
\rule[-2ex]{0pt}{4ex} and natural transformations & $H_!F\to G$.
\end{tabular}
\end{center}
The right Kan extension of $F$ along $H$ is the functor $H_*F$ such that there is a bijection between
\begin{center}
\begin{tabular}{lc}
\rule[-2ex]{0pt}{4ex} natural transformations & $G\circ H\to F$, \\
\rule[-2ex]{0pt}{4ex} and natural transformations & $G\to H_*F$.
\end{tabular}
\end{center}
If $\cB=\cC$ and $F=H$, the universal property of Kan extensions turn $F_!F$ into a comonad and $F_*F$ into a monad.
We shall put $\END^\vee(F)=F_!F$ and $\END(F)=F_*F$ and call them respectively the comonad of coendomorphisms of $F$ and the monad of endomorphisms of $F$. They have the following universal properties.
If $G$ is a comonad, there is a bijection between
\begin{center}
\begin{tabular}{lc}
\rule[-2ex]{0pt}{4ex} left $G$-comodule structure & $F\to G\circ F$, \\
\rule[-2ex]{0pt}{4ex} and comonad morphisms & $\END^\vee(F)\to G$.
\end{tabular}
\end{center}
If $G$ is a monad, there is a bijection between
\begin{center}
\begin{tabular}{lc}
\rule[-2ex]{0pt}{4ex} left $G$-module structures & $G\circ F\to F$, \\
\rule[-2ex]{0pt}{4ex} and monad morphisms & $G\to \END(F)$.
\end{tabular}
\end{center}

We shall be interested in the case where $\cA=\un$ is the category with a single objet with the monoidal unit of $\cV$ as object of endomorphisms.
A functor $\un\to \cB$ is simply an object $X$ of $\cB$. $\END^\vee(X)$ and $\END(X)$ are the comonad of coendomorphisms of $X$ and the
monad of endomorphisms of $X$. 
If $G$ is a comonad, a $G$-coalgebra structure on $X$ is the same thing as a comonad morphism $\END^\vee(X)\to G$
and if $G$ is a monad, a $G$-algebra structure on $X$ is the same thing as a monad morphism $G\to \END(X)$.

\begin{rem}
In case $\cA=\cV^A$, $\cB=\cV^B$, $\cC=\cV^C$ and $F$, $G$ and $H$ are analytic functors, there exists an analytic right Kan extension $[\HH,F]$ of $F$ along $H$. 
Recall our definition of $\HH_{\overline{x}}^{\overline{z}} = \left(\otimes_i H_{x_i}^{\overline{u_i}}\right)S(X)_{\overline{u_1}\dots \overline{u_n}}^{\overline{z}}$, then, we have bijection between
\begin{center}
\begin{tabular}{lc}
\rule[-2ex]{0pt}{4ex} natural transformations & $(G\circ H)^{\overline{z}}_y = G^{\overline{x}}_y \HH_{\overline{x}}^{\overline{z}} \to F^{\overline{z}}_y$, \\
\rule[-2ex]{0pt}{4ex} and natural transformations & $G^{\overline{x}}_y\to (H_*F)^{\overline{x}}_y = \left[ \HH_{\overline{x}}^{\overline{z}}, F^{\overline{z}}_y \right]$.
\end{tabular}
\end{center}
In the particular case where $F=H:\un\to \cV^B$ is a $B$-colored object $X$, the formula gives $\End(X)^{\overline{x}}_y = \left[ X_{\otimes \overline{x}},X_y \right]$, the operad of endomorphisms of $X$. 
The previous considerations give that, if $P$ is a $B$-colored operad, a $P$-algebra structure on $X$ is the same thing as a operad morphism $P\to \End(X)$.

\medskip
The functor $[\HH,F]$ is different from $H_*F$ since it represents the functor $G\mapsto Nat(G\circ H,F)$ only when $G$ is analytic
but there is a natural transformation $[\HH,F]\to H_*F$ which turns $[\HH,F]$ into an analytic coreflection of $H_*F$, that is, for $G$ analytic, we have a bijection between
\begin{center}
\begin{tabular}{lc}
\rule[-2ex]{0pt}{4ex} natural transformations & $G\to [\HH,F]$\\
\rule[-2ex]{0pt}{4ex} and natural transformations & $G\to H_*F$.
\end{tabular}
\end{center}

The replacement of the monad $\END(X)$ by the operad $\End(X)$ simplifies the study of algebras over operads.
The notion of coalgebra can be defined using some operad ${\rm co}\End(X)$ (see \cite{LV}) but we shall see in the next section that, for our purposes, it is the comonad $\END^\vee(X)$ which will be of use.

\end{rem}


\subsection{Lax comonads and their coalgebras}\label{laxcomonad}

%
%
%
%
%
%

Recall that $\Delta_+$, the category of finite (possibly empty) ordinals is a monoidal category for the ordinal sum. 
Recall also that, if $(C,\otimes)$ is a monoidal category, a monoidal functor $\Delta_+\to (\cD,\otimes)$ is the data of a monoid in $C$ and 
that a monoidal functor $\Delta_+^{op}\to (\cD,\otimes)$ is the data of a comonoid in $C$. We shall cal a {\em lax comonoid} a lax monoidal functor $\Delta_+^{op}\to (\cD,\otimes)$. A morphism of lax comonoids is a lax monoidal natural transformation.
The category of comonoids embeds fully faithfully in the category of lax comonoids.
In case $\cD=(END(\cC),\circ)$ is the category of endofunctors of a category $\cC$, we shall call a lax monoidal functor $\Delta_+\to END(\cC)$ a {\em lax comonad}.

\medskip
Recall the functors $\pi$ and $\epsilon$ from section \ref{piepsilon}.
Postcomposing by $\pi$ and precomposing by $\epsilon$ creates a functor $END(\cV^{S(X)})\to END(\cV^X)$ which has no compatibility in general with the monoidal structures (we would need for that a comparison between $\epsilon\pi $ and the identity but there is none).
Therefore, the image of a monoid or a comonoid in $END(\cV^{S(X)})$ is not in general a monoid of a comonoid in $END(\cV^{X})$.
However, for the functors $\PP$ and $\PP^\wedge$ we can say something.

The formula $\pi\PP^n\epsilon =(P^\sim)^n$ says precisely that the monad $\PP$ is send to a monad $P^\sim$, \ie that the composite functor
$$
\begin{array}{ccccc}
\Delta_+ &\tto & [\cV^{S(X)},\cV^{S(X)}] &\tto & [\cV^X,\cV^X]\\
n &\mto & \PP^n &\mto & \pi\PP^n\epsilon
\end{array}
$$
is a monoidal functor.
The situation is more complex for the functors $P^\wedge$ and $\PP^\wedge$ but it justifies the definition of a lax comonad.
\begin{prop}\label{laxcomonP}
The maps $\alpha_{n,m}:\pi(\PP^\wedge)^n\epsilon\pi(\PP^\wedge)^m\epsilon\to \pi(\PP^\wedge)^{n+m}\epsilon$
endow the composite functor 
$$
\begin{array}{ccccc}
\Delta_+^{op} &\tto & [\cV^{S(X)},\cV^{S(X)}] &\tto & [\cV^X,\cV^X]\\
n &\mto & (\PP^\wedge)^n &\mto & \pi(\PP^\wedge)^n\epsilon
\end{array}
$$
with a lax monoidal structure.
\end{prop}
\begin{proof}
The condition on the unit is trivial since $\pi (Id) \epsilon = Id$. 
We need only to prove the associativity condition
$$\xymatrix{
\pi(\PP^\wedge)^n\epsilon\pi(\PP^\wedge)^m\epsilon\pi(\PP^\wedge)^\ell\epsilon \ar[r]\ar[d] & \pi(\PP^\wedge)^{n+m}\epsilon\pi(\PP^\wedge)^\ell\epsilon \ar[d] \\
\pi(\PP^\wedge)^n\epsilon\pi(\PP^\wedge)^{m+\ell}\epsilon \ar[r] & \pi(\PP^\wedge)^{n+m+\ell}\epsilon.
}$$
This is a straightforward computation, essentially the argument reduces to the commutation of
$$\xymatrix{
\otimes_i[A_i,\otimes_{j_i}[B_{j_i},C_{j_i}]] \ar[r]\ar[d] &  [\otimes_iA_i,\otimes_i\otimes_{j_i}[B_{j_i},C_{j_i}]] \ar[d] \\
\otimes_i[A_i,[\otimes_{j_i}B_{j_i},\otimes_{j_i}C_{j_i}]]  \ar[r] & [\otimes_iA_i,[\otimes_i\otimes_{j_i}B_{j_i},\otimes_i\otimes_{j_i}C_{j_i}]].
}$$
\end{proof}

The functor of proposition \ref{laxcomonP} is then a lax comonad.
With a slight abuse in notation, we shall call again $P^\wedge$ this functor.
If $Q$ is a lax comonad, we shall use the classical convention for simplicial objects and denote is components by $Q_n$. 
In particular, the functors $\pi(\PP^\wedge)^n\epsilon$ shall be abbreviated $P^\wedge_n$.

\bigskip
We now establish a few results on lax comonads and their coalgebras.

\medskip
Let $Q$ be a lax comonad on a category $\cC$, a {\em coalgebra over $Q$} is defined as an object $C\in \cC$ together with a map $\delta_C:C\to Q_1(C)$ such that
\begin{itemize}
\item (Unitality) the composition $C\to Q_1(C)\to C$ is the identity of $C$
\item (Coassociativity) the following diagram commutes
$$\xymatrix{
C\ar[rr]^-{\delta}\ar[d]_-{\delta} && Q_1(C)\ar[rr]^-{Q_1(\delta)} && Q_1Q_1(C)\ar[d]^{\alpha_{1,1}}\\
Q_1(C)\ar[rrrr]_{\Delta(C)} && && Q_2(C)
}$$
\end{itemize}

If $C$ and $D$ are two $Q$-coalgebras, a morphism of $Q$-coalgebras is defined as a map $f:C\to D\in \cC$ such that the following diagram commutes
$$\xymatrix{
C\ar[rr]^-{\delta_C} \ar[d]_f&& Q_1(C)\ar[d]^-{Q_1(f)}\\
D\ar[rr]_-{\delta_D} && Q_1(D)
}$$

We shall say that a $Q$-coalgebra $C$ is {\em cofreely generated by $C\to V$} (or simply {\em cofree}) if 
for any $Q$-coalgebra $D$, the maps $C\to V$ induces a bijection between
\begin{center}
\begin{tabular}{lc}
\rule[-2ex]{0pt}{4ex} $Q$-coalgebra morphisms & $D\to C$\\
\rule[-2ex]{0pt}{4ex} and maps in $\cV^X$  & $D\to V$.
\end{tabular}
\end{center}

The following lemma is a direct consequence of the caracterisation of $P$-coalgebras of section \ref{coalgebras}. It also justifies the previous definition.
\begin{lemma}
A $P$-coalgebra structure on $C$ is the same thing as a coalgebra structure on the lax comonad $P^\wedge$.
\end{lemma}

\medskip
Recall $\END^\vee(C)$, the comonad of coendomorphisms of $C$, from section \ref{coendomorphisms}.
\begin{lemma}\label{coendcoalg}
Let $Q$ be a lax comonad, then a  $Q$-coalgebra structure on $C$ is the same thing as a lax comonad morphism $\END^\vee(C)\to Q$.
\end{lemma}
\begin{proof}
A lax comonad morphism $\END^\vee(C)\to Q$ is a family of natural transformations $\gamma_n:\END^\vee(C)^n\to Q_n$ such that the squares
$$\xymatrix{
\END^\vee(C)^n\circ \END^\vee(C)^m\ar@{=}[d]\ar[rr]&& Q_n \circ Q_m\ar[d]\\
\END^\vee(C)^{n+m} \ar[rr]&& Q_{n+m}
}$$
commute.

The map $C\to Q_1(C)$ give a map $\END^\vee(C)\to Q_1$ from which we deduce maps $\gamma_n:\END^\vee(C)^n\to (Q_1)^n\to Q_n$. 
It is clear that the above square commute, we need only to check that the $\gamma_n$ are natural in $n$.
Decomposing arrows in $\Delta_+$ in faces and degeneracies, the only non-trivial condition is the commutation of the square
$$\xymatrix{
\END^\vee(C)\ar[d]\ar[rr]^-{\gamma_1}&& Q_1\ar[d]\\
\END^\vee(C)^2 \ar[r]&Q_1Q_1\ar[r]& Q_2
}$$
which is a consequence of the coassociativity condition on $C$.
\end{proof}

\medskip
Let $\cD$ be a monoidal category. We shall say that a lax comonoid $Q:\Delta_+\to \cD$ has a {\em coreflection into comonoids} is there exist a comonoid $Q'$ and a lax comonoid morphism $Q'\to Q$ which induces, for any comonoid $R$, a bijection between lax comonoid morphisms $R\to Q$ and comonoid morphisms $R\to Q'$.

\begin{prop}\label{propcoref}
Let $Q$ be a lax comonad on a category $\cC$ having a coreflection $Q'$ into comonads, then
\begin{enumerate}
\item there is an equivalence of categories between $Q$-coalgebras and $Q'$-coalgebras;
\item the cofree $Q$-coalgebras coincide with the cofree $Q'$-coalgebras;
\item and the category of $Q$-coalgebras is comonadic over $\cC$.
\end{enumerate}
\end{prop}
\begin{proof}
(1) Lemma \ref{coendcoalg} proves that coassociative maps $C\to Q(C)$ and $C\to Q'(C)$ are in bijection. 
We need only to prove that the notion of morphisms coincides, but this is because $Q'\to Q$ is a natural transformation.

(2) is a consequence of (1).

(3) By (2) the forgetful functor $U:Q\Coalg\to \cC$ has a right adjoint given by $Q'$, the comonadicity is then the statement (1).
\end{proof}

\medskip
Proposition \ref{propcoref} applies in particular to the lax comonad $P^\wedge$ of an operad $P$. The purpose of the next chapter is to prove that $P^\wedge$ admits a comonadic coreflection.

\subsection{The coreflection theorem}\label{sectioncoreflection}

We are going to construct the comonadic coreflection $P^\vee$ of the lax comonad $P^\wedge$ recursively. 
The method will use specific properties of $P^\wedge$ as well as conditions of the monoidal category $\cV$ and seems difficult to generalize although the principle guiding it looks quite general.
%
%

Let us consider the following diagram whose terms we are going to explain.
\begin{align}
\tag{D}\label{maindiagram}
\vcenter{
\xymatrix{
P^\vee\ar@{-->}[d]\ar@{-->}[r]&(P^\vee)^2\ar@{-->}[d]\ar@<.6ex>@{-->}[r]\ar@<-.6ex>@{-->}[r]&(P^\vee)^3\ar@{-->}[d]\ar@<1.2ex>@{-->}[r]\ar@{-->}[r]\ar@<-1.2ex>@{-->}[r]&(P^\vee)^4\ar@{-->}[d]\ar@{-->}@<1.8ex>[r]\ar@{-->}@<.6ex>[r]\ar@{-->}@<-.6ex>[r]\ar@{-->}@<-1.8ex>[r] &(P^\vee)^5\ar@{-->}[d] & \dots \\
\vdots\ar@{..>}[d]&\vdots\ar@{-->}[d]&\vdots\ar@{-->}[d]&\vdots\ar@{-->}[d]&\vdots\ar@{-->}[d]&\\
Q_4\ar@{..>}[d]\ar@{..>}[r]&(Q_3)^2\ar@{-->}[d]\ar@<.6ex>@{-->}[r]\ar@<-.6ex>@{-->}[r]\ar@{}[rd]|{(4'')}&(Q_2)^3\ar@{-->}[d]\ar@<1.2ex>@{-->}[r]\ar@{-->}[r]\ar@<-1.2ex>@{-->}[r]\ar@{}[rd]|{(4')}&(Q_1)^4\ar@{-->}[d]\ar@{-->}@<1.8ex>[r]\ar@{-->}@<.6ex>[r]\ar@{-->}@<-.6ex>[r]\ar@{-->}@<-1.8ex>[r] \ar@{}[rdddd]|{(4)}&(P^\wedge)^5\ar[dddd] & \dots\\
Q_3\ar@{..>}[d]\ar@{..>}[r]&(Q_2)^2\ar@{-->}[d]\ar@<.6ex>@{-->}[r]\ar@<-.6ex>@{-->}[r]\ar@{}[rd]|{(3')}&(Q_1)^3\ar@{-->}[d]\ar@<1.2ex>@{-->}[r]\ar@{-->}[r]\ar@<-1.2ex>@{-->}[r]\ar@{}[rddd]|{(3)}&(P^\wedge)^4\ar[ddd]&&\\
Q_2\ar@{..>}[d]_{\chi_2}\ar@{..>}[r]^{\delta_2}&(Q_1)^2\ar@{-->}[d]^{(\chi_1)^2}\ar@<.6ex>@{-->}[r]\ar@<-.6ex>@{-->}[r]\ar@{}[rdd]|{(2)}&(P^\wedge)^3\ar[dd]& & \\
Q_1\ar@{..>}[d]_{\chi_1}\ar@{..>}[r]^{\delta_1}&(P^\wedge)^2 \ar[d]_{\alpha_{1,1}}& & &\\
P^\wedge\ar[r]_{\Delta} & P^\wedge_2\ar@<.6ex>[r]^{\Delta^{(1)}}\ar@<-.6ex>[r]_{\Delta^{(2)}} & P^\wedge_3\ar@<1.2ex>[r]\ar[r]\ar@<-1.2ex>[r] & P^\wedge_4\ar@<1.8ex>[r]\ar@<.6ex>[r]\ar@<-.6ex>[r]\ar@<-1.8ex>[r]& P^\wedge_5 & \dots
}}
\end{align}

The plain arrows are the structure maps of $P^\wedge$. 
The bottom row is the simplicial diagram of $P^\wedge$ but we have represented only the degeneracy maps (the diagonals of the comonad structure). 
The vertical plain arrows are the lax structure of $P^\wedge$.
The dotted arrows and the objects $Q_k$ will be defined by some limits. 
Finally, the dashed arrows will be constructed from the dotted ones.

\medskip

The functor $P^\wedge$ has a counit $P^\wedge \to Id$ but is missing a diagonal.
The construction we are going to describe focus on how to build such a diagonal. The counit will be dealt with afterwards.
The general idea is the following. We would like to have a diagonal map $P^\wedge\to P^\wedge P^\wedge$ lifting the diagonal $\Delta : P^\wedge\to P^\wedge_2$
$$\xymatrix{
&(P^\wedge)^2 \ar[d]^{\alpha_{1,1}} &\\
P^\wedge\ar[r]_{\Delta} \ar@{-->}[ru]^-{?}& P^\wedge_2.
}$$
We can consider the universal domain for the existence of such a lift, that is the fiber product $Q_1$ of  $(P^\wedge)^2\to P^\wedge_2 \leftarrow P^\wedge$,
but this is not enough since we want the diagonal $Q_1\to (P^\wedge)^2$ to lift into $Q_1\to (Q_1)^2$
$$\xymatrix{
&(Q_1)^2\ar[d]\\
Q_1\ar@{-->}[ru]^-{?}\ar[r]^-{\delta_1}\ar[d]&(P^\wedge)^2 \ar[d]^{\alpha_{1,1}} &\\
P^\wedge\ar[r]_{\Delta} & P^\wedge_2.
}$$
So we consider $Q_2$ the fiber product of $(Q_1)^2\to (P^\wedge)^2 \leftarrow Q_1$ and etc. ad infinitum.
Let $Q_\infty$ be the limit of the $Q_k$s, it is unfortunately false that the diagonal maps $Q_{k+1}\to (Q_k)^2$ will define a diagonal $\delta:Q_\infty\to (Q_\infty)^2$. But this will be true if $(Q_\infty)^2$ is proven to be the limit of the tower of $(Q_k)^2$. This will be the main issue at hand, addressed in lemma \ref{mainlemma}.

Another issue is that the diagonal thereby created has no reason to be coassociative. To ensure this fact we need to defined $Q_2$ as the limit of the bigger diagram
$$\xymatrix{
&(Q_1)^2\ar[d]_{(\chi_1)^2}\ar@<.6ex>[rr]^{\delta_1\chi_1}\ar@<-.6ex>[rr]_{\chi_1\delta_1}&&(P^\wedge)^3\ar[dd]\\
Q_1\ar[d]_{\chi_1}\ar[r]^{\delta_1}&(P^\wedge)^2 \ar[d] &\\
P^\wedge\ar[r]_{\Delta} & P^\wedge_2\ar@<.6ex>[rr]^{\Delta^{(1)}}\ar@<-.6ex>[rr]_{\Delta^{(2)}} && P^\wedge_3;
}$$
and $Q_3$ as the limit of 
$$\xymatrix{
&(Q_2)^2\ar[d]\ar@<.6ex>[r]\ar@<-.6ex>[r]&(Q_1)^3\ar[d]\ar@<1.2ex>[r]\ar[r]\ar@<-1.2ex>[r]&(P^\wedge)^4\ar[ddd]\\
Q_2\ar[d]_{\chi_2}\ar[r]^{\delta_2}&(Q_1)^2\ar[d]\ar@<.6ex>[r]\ar@<-.6ex>[r]&(P^\wedge)^3\ar[dd]\\
Q_1\ar[d]_{\chi_1}\ar[r]^{\delta_1}&(P^\wedge)^2 \ar[d]& \\
P^\wedge\ar[r]_{\Delta} & P^\wedge_2\ar@<.6ex>[r]^{\Delta^{(1)}}\ar@<-.6ex>[r]_{\Delta^{(2)}} & P^\wedge_3\ar@<1.2ex>[r]\ar[r]\ar@<-1.2ex>[r] & P^\wedge_4
}$$
etc.\footnote{In fact, the diagram defining $Q_k$ can be simplified (see remark \ref{rem2trunc})}
This way, at the limit, we will have a full simplical structure on the $(Q_\infty)^n$ (taking into account the non-drawn face maps), turning $Q_\infty$ into a comonad. Moreover, by construction, we will have a canonical morphism of lax comonads $Q_\infty\to P^\wedge$. 
Proposition \ref{mainthm1} will prove that this morphism is the coreflexion.

The last issue is to construct the arrows of the diagrams defining the $Q_i$s, that is the dashed arrows of the diagram (\ref{maindiagram}).
We put $Q_0=P^\wedge$ and we call $\chi_n$ the canonical map $Q_n\to Q_{n-1}$. 
Then, the dashed vertical arrows $(Q_n)^k\to (Q_{n-1})^k$ are simply tensor products (for the composition of functors) of $\chi_n$s.
The dashed horizontal arrows, named $\delta_k^{(n,i)}:(Q_k)^n\to (Q_{k-1})^{n+1}$, are constructed from the two maps $\chi_k:Q_k\to Q_{k-1}$ and $\delta_k:Q_k\to (Q_{k-1})^2$ by the formula
$$
\delta_k^{(n,i)}:=\chi_k\dots \delta_k\dots \chi_k:(Q_k)^n\to (Q_{k-1})^{n+1}
$$
where $\delta_k$ is applied on the $i$-th factor.

Each $\delta_k^{(n,i)}$ is defined over a degeneracy map $\Delta^{(i)}:P^\wedge_n\to P^\wedge_{n+1}$.
The following lemmas ensure that the dashed arrows commute with the plain arrows.

\begin{lemma}\label{laxmap1}
The following squares (denoted $(n)$ in the diagram) commute
$$\xymatrix{
(Q_1)^n\ar[rr]^{\delta_1^{(n,i)}} \ar[d]_{\alpha(\chi_1)^n}&& (P^\wedge_1)^{n+1}\ar[d]^-{\alpha_{1,\dots,1}}\\
P^\wedge_n\ar[rr]^{\Delta^{(i)}} && P^\wedge_{n+1}.
}$$
\end{lemma}
\begin{proof}
We can factor these squares as 
$$\xymatrix{
(Q_1)^{i-1}Q_1(Q_1)^{n-i}\ar[d]_{(\chi_1)^n}\ar[rrr]^{(\chi_1)^{i-1}\delta_1(\chi_1)^{n-i}}&&&(P^\wedge)^{i-1}(P^\wedge)^2(P^\wedge)^{n-i} \ar[d]^{\alpha(\alpha_{1,1})\alpha}\\
(P^\wedge)^{i-1}P^\wedge(P^\wedge)^{n-i} \ar[d]_{\alpha}\ar[rrr]_{\alpha\Delta\alpha} &&& (P^\wedge)^{i-1}P^\wedge_2(P^\wedge_1)^{n-i} \ar[d]^{\alpha}\\
P^\wedge_n\ar[rrr] &&&P^\wedge_{n+1}.
}$$
The commutation of the top square reduces to the definition of $Q_1$ and that of the bottom square is the lax structure of $Q$.
\end{proof}


\begin{lemma}\label{laxmap2}
The following squares (denoted $(n')$ and $(n'')$ in the diagram) commute
$$\xymatrix{
(Q_{k+1})^n\ar[rr]^{\delta_k^{(n,i)}} \ar[d]_{(\chi_{k+1})^n}&& (Q_k)^{n+1}\ar[d]^-{(\chi_{k})^{n+1}}\\
(Q_{k})^n\ar[rr]^{\delta_{k-1}^{(n,i)}} && (Q_{k-1})^{n+1}.
}$$
\end{lemma}
\begin{proof}
Easy consequence of the definition of $Q_{k+1}$.
\end{proof}

\begin{rem}\label{rem2trunc}
The diagram defining $Q_{k+1}$ can be reduced to the coinitial diagram
$$\xymatrix{
&(Q_k)^2\ar[d]\ar@<.6ex>[r]\ar@<-.6ex>[r]&(Q_{k-1})^3\\
Q_k\ar[r]&(Q_{k-1})^2.
}$$
Somehow, this is a way to say that the higher associativity conditions are implied by the associativity.
Therefore, we could have restricted our simplicial diagrams to 2-truncated simplicial diagrams, even for the diagram (\ref{maindiagram}), to defined $Q_\infty$ and its structure.
However, doing so, it would not have been trivial that the maps $Q_\infty^n\to P^\wedge_n$ are natural transformations of simplicial lax monoidal functor.
\end{rem}

\begin{rem}\label{remsmith}
The recursion used in \cite{Smith} is not the same as our. In this reference, the equivalent $L_k$ of our $Q_k$ are defined by cartesian squares
$$\xymatrix{
L_{k+1} \ar[r]\ar[d]&P^\wedge L_k\ar[d]\\
L_k\ar[r]&P^\wedge_2.
}$$
However, we have $Q_1=L_1$ and this will be useful in remark \ref{remsmithfinal}.
\end{rem}

\bigskip

We now turn to the main results of the section. 
Lemma \ref{mainlemma} which is the main point of our proof requires some hypothesis that we state before.

\medskip
We shall call a {\em countable intersection}, the data of a diagram $\dots \to A_{k+1}\to A_k \to \dots A_0$ indexed by $\NN$ where all maps are monomorphisms.

\begin{hypo}\label{hypo}
We are going to assume the following property on the monoidal category $(\cV,\otimes)$:
\begin{itemize}
\item the canonical natural transformation $[A,A']\otimes [B,B']\to [A\otimes B,A'\otimes B']$ is a monomorphism,
\item the functor $\otimes$ commute with countable intersections in each variable.
\end{itemize}
\end{hypo}

These hypothesis are satisfied for the following monoidal categories:
the cartesian category $\Set$ of sets, any topos, and in fact any cartesian category (for example CGH spaces);
the category $\Vect$ of vector spaces,  
and the category $\dgVect$ of chain complexes over a field.

Condition (2) is trivial in cartesian categories, here is a proof for $\Vect$ and $\dgVect$.
\begin{lemma}\label{monotens}
Let $A$ be an object in $\Vect$ (or $\dgVect$), then the functor $V\mapsto A\otimes V$ commutes to countable intersections.
\end{lemma}
\begin{proof}
Let $\dots \to V_{k+1}\to V_k\to \dots V_0$ be a countable intersection of subobjects of an object $V_0$ and let $V_\infty$ be the limit of the intersection.
Let $V'_k$ be a complement of $V_{k+1}$ in $V_k$, we have $V_k=V_\infty \oplus \bigoplus_{i > k} V'_i$.
We deduce that $A\otimes V_k = A\otimes V_\infty \oplus \bigoplus_{i > k} A\otimes V'_i$ and $\lim (A\otimes V_k) = A\otimes V_\infty$.
The proof can be adapted for dg-vector spaces by remarking that intersections of complexes are computed degreewise on the underlying graded objects.
\end{proof}

Hypothesis (2) will be useful because of the following fact.

\begin{lemma}\label{monotens2}
Let $(\cC,\otimes)$ be a symmetric monoidal category, then the functor $V\mapsto V^{\otimes n}$ commutes to all sifted limits preserved by $\otimes$ in each of its variables.
\end{lemma}
\begin{proof}
We prove the result for $n=2$, the general proof is similar.
Let $I$ be the opposite of sifted category and $V:I\to \cC$ a contravariant diagram.
We assume that, for any $A\in \cC$, the functor $A\otimes -$ commutes to the limit of $V$.
Then we have $(\lim_i V_i)\otimes (\lim_i V_i) = \lim_{i,j} V_i\otimes V_j$.
Because $I^{op}$ is sifted we have also $\lim_{i,j} V_i\otimes V_j = \lim_i V_i\otimes V_i$.
This proves that $\lim_i (V_i\otimes V_i) = (\lim_i V_i)\otimes (\lim_i V_i)$.
\end{proof}
Since $\NN$ is filtered, hence sifted, hypothesis (2) implies that the functors $V\mapsto V^{\otimes n}$ preserve countable intersections.

\medskip
We need a few more lemmas in order to prove that $Q_\infty$ is the comonadic coreflection of $P^\wedge$.

\begin{lemma}\label{Qmono}
\begin{enumerate}
\item All $Q_k$ preserve countable intersections.
\item All maps $\chi_k:Q_k\to Q_{k-1}$ are monomorphisms.
\end{enumerate}
\end{lemma}
\begin{proof}
We are going to prove both statements with an induction.
We put $Q_0 = P^\wedge$ and $Q_{-1} = P^\wedge_2$ and start at $k=0$.

(1) Using lemma \ref{monotens2}, we deduce that $Q_0$ preserves countable intersections.
(2) We have $(Q_0)^2(V) = [P_y^{\overline{x}}, \otimes_i [P_{x_i}^{\overline{z_i}}, V^{\otimes \overline{z_i}}]$
and $Q_{-1}(V) = [P_y^{\overline{x}},[P_{\overline{x}}^{\overline{z}}, V^{\otimes \overline{z}}]]$.
$\otimes_i [P_{x_i}^{\overline{z_i}}, V^{\otimes \overline{z_i}}]\to [P_{\overline{x}}^{\overline{z}}, V^{\otimes \overline{z}}]$ is a monomorphism by hypothesis \ref{hypo}.
Ends preserves monos in their second variable, hence $(Q_0)^2\to Q_{-1}$ is a monomorphism.

Now, we assume that the statements are true at rank $\leq k$ for some $k$. 
By remark \ref{rem2trunc}, $Q_{k+1}$ is the limit of the diagram
$$\xymatrix{
&(Q_k)^2\ar[d]\ar@<.6ex>[r]\ar@<-.6ex>[r]&(Q_{k-1})^3\\
Q_k\ar[r]&(Q_{k-1})^2.
}$$
(1) By hypothesis, $Q_k$, $Q_{k-1}$ and their composites preserve countable intersections.
The limit product of functors preserving some limits also preserves those limits since fiber products are computed termwise in functor categories.
This proves that $Q_{k+1}$ preserves countable intersections.
(2) Let $Q'_{k+1}$ be the pullback of $(Q_{k})^2\to (Q_{k-1})^2 \leftarrow Q_{k}$. 
The map $Q'_{k+1}\to Q_k$ is a mono as the pullback of the mono $(Q_{k})^2\to (Q_{k-1})^2$.
Then $Q_{k+1}$ can equivalently be defined as the equalizer of $Q'_{k+1}\to (Q_{k})^2\rightrightarrows (Q_{k-1})^2$, hence 
$Q_{k+1}\to Q'_{k+1}$ is a mono and so is $Q_{k+1}\to Q_{k}$ by composition.
\end{proof}
In fact, we have proved that the $Q_k$s preserve all limits preserved by $P^\wedge$ and $P^\wedge_2$.

\begin{lemma}\label{mainlemma}
Under hypothesis \ref{hypo}, $(Q_\infty)^n$ is the limit of the tower of $(Q_k)^n$.
\end{lemma}
\begin{proof}
We are going to prove the case $n=2$ to simplify notations, the general proof is similar.
By lemma \ref{Qmono}, we deduce that $\lim Q_k$ is a countable intersection of functors and that 
$$(Q_\infty)^2 = (\lim Q_\ell) (\lim Q_k) = \lim Q_\ell(\lim Q_k)= \lim_{k,\ell} Q_\ell Q_k.$$
The category $\NN$ indexing the countable intersection is sifted,
we deduce that $\lim_{k,\ell} Q_\ell Q_k = \lim_{k} Q_k Q_k$, which proves the result.
\end{proof}

\begin{lemma}\label{comonadic}
$Q_\infty$ is a comonad.
\end{lemma}
\begin{proof}
From lemma \ref{mainlemma} we deduce that the maps $(Q_k)^n\to (Q_{k-1})^{n+1}$ lift to maps $(Q_\infty)^n\to (Q_\infty)^{n+1}$.
Moreover the definition of the $(Q_k)^n\to (Q_{k-1})^{n+1}$ ensure that all simplicial identities are statisfied. 
This prove that $Q_\infty$ has a coassociative codiagonal. 
Now, we need to check that is has a compatible counit.

$P^\wedge$ has a counit $\epsilon:P^\wedge \to Id$ from which we define maps $\epsilon_k:Q_k\to P^\wedge \to Id$ ($k=0, \dots ,\infty$).
Then, from the $\epsilon_k$s, we can construct maps $\epsilon_k^{(i)}:(Q_k)^n\to (Q_k)^{n-1}$ by applying $\epsilon_k$ to the $i$-th factor. 
Using inductively the definition of $Q_{k}$, we can lift the maps $\epsilon_k^{(i)}$ into maps
$\eta_k^{(i)}:(Q_k)^n\to (Q_{k+1})^{n-1}$ and complete the dashed arrows of the diagram (\ref{maindiagram}) into augmented truncated simplicial diagrams.
Using lemma \ref{Qmono} we prove that the limits of the maps $\eta_k^{(i)}$ are the maps $\epsilon_\infty^{(i)}$ and we get a full augmented simplicial diagram on the $(Q_\infty)^n$s.
\end{proof}

\begin{thm}\label{mainthm1}
$Q_\infty$ is the comonadic coreflexion of $P^\wedge$.
\end{thm}

Applying proposition \ref{propcoref}, we get the final result.
\begin{cor}
$Q_\infty$ is the cofree $P$-coalgebra comonad.
\end{cor}

\begin{proof}[Proof of \ref{mainthm1}]

Let us first prove that there is a canonical morphism of lax comonads to $Q_\infty\to P^\wedge$.
The construction of $Q_\infty$ has produced maps $(Q_\infty)^n\to (P^\wedge)^n\to P^\wedge_n$.
Recall from the proof of lemma \ref{comonadic} the simplicial diagram giving the comonad structure of $(Q_\infty)^n$.
Lemma  \ref{laxmap1} proves that the composition $(Q_\infty)^n\to P^\wedge_n$ are natural with respect to simplicial degeneracies maps.
An analogous proof would show that they are also compatible with the faces maps.
This construct a morphism of lax comonad $Q_\infty\to P^\wedge$.
Let us prove that it is a coreflection. 

Let $R$ be a comonad with a lax comonad morphisms $R\to P^\wedge$.
In particular this data gives a diagram
$$
\xymatrix{
R\ar[dd]\ar[r]&R^2\ar[d]\\
 & (P^\wedge)^2\ar[d]\\
P^\wedge\ar[r] & P^\wedge_2
}$$
and a map $R\to Q_1$ factorizing $R\to P^\wedge$.
From this map we get a commutative diagram
$$\xymatrix{
R\ar[dd]\ar[r]&R^2\ar[d]\ar@<.6ex>[rr]\ar@<-.6ex>[rr]&&R^2\ar[d]\\
&(Q_1)^2\ar[d]\ar@<.6ex>[rr]\ar@<-.6ex>[rr]&&(P^\wedge)^3\ar[dd]\\
Q_1\ar[d]\ar[r]&(P^\wedge)^2 \ar[d] &\\
P^\wedge\ar[r] & P^\wedge_2\ar@<.6ex>[rr]\ar@<-.6ex>[rr] && P^\wedge_3;
}$$
and thus a map $R\to Q_2$ factorizing $R\to P^\wedge$.
Continuing this way, we construct a factorisation $R\to Q_\infty\to P^\wedge$.
By construction the map $R\to Q_\infty$ extends to a map of augmented simplicial diagrams, that is a map of comonoids.
This proves that $Q_\infty\to P^\wedge$ is the comonadic coreflection.
\end{proof}

\section{Operadic representative functions}

In this section, the basic symmetric monoidal category $\cV$ is assumed to be $\Vect$ or $\dgVect$. 
In both cases the unit of the tensor product is the ground field $k$. 
If $V\in \cV$, we shall denote by $V^\star=[V,k]$ the dual of $V$ and by $V^{\star\star}$ the double dual of $V$.
The canonical map $[A,B]\otimes [A',B']\to [A\otimes A',B\otimes B']$ 
induces maps $A^\star \otimes (A')^\star\to (A\otimes A')^\star$ and 
maps $A^\star\otimes B'\to [A,B']$.

The main consequences of the hypothesis on $\cV$ are given in the following lemmas. 
In the whole section, we shall give proofs only in the case of $\Vect$, they can be adapted to $\dgVect$ with minor changes to take care of the differential and the grading. Also the notion corresponding to finite dimension vector spaces has to be complexes whose total space are finite dimensional.

\begin{lemma}\label{liftlemma}
\begin{enumerate}
\item A map $f:V_1\otimes V_2\to k$ is in $(V_1)^\star \otimes (V_2)^\star$ iff it factors by a finite dimensional space.
\item A map $f:V_1\to V_2$ is in $(V_1)^\star\otimes V_2$ iff it factors through a finite dimensional space.
\end{enumerate}
\end{lemma}
\begin{proof}

(1) Let us suppose that  $f:V_1\otimes V_2\to k$ factors through a finite dimensional space $W$ and a map $g:W\to k$.
We can assume that $W$ is a tensor product of $W_1\otimes W_2$ where $W_i$ is a finite dimensional quotient of $V_i$.
Let $(\alpha_i^j)$ be a basis of $(W_i)^\star$, then $g= \sum_{i,j} g_{i,j} \alpha_1^i\otimes\alpha_2^j$ for some coefficients $g_{i,j}\in k$.
Let us call $\beta_i^j$ the image of $\alpha_i^j$ by $(W_i)^\star\to (V_i)^\star$, then $f= \sum_{i,j} g_{i,j} \beta_1^i\otimes \beta_2^j$.
Reciprocally, if $f\in (V_1)^\star\otimes (V_2)^\star$, it can be written as $f= \sum_{j\in J} \gamma_1^j\otimes \gamma_2^j$ where $\gamma_i^j\in V_i$ and $J$ is finite. The collection of $\gamma_i^j$ for $i$ fixed define a map $V_i\to k^J$, then $f$ factors through $V_1\otimes V_2\to k^J\otimes k^J$.

(2) Let us suppose that $f:V_1\to V_2$ facror through a finite dimensional space $W$ and a map $g:W\to V_2$.
 We can assume $W$ is a quotient of $V_1$.
Let $(\alpha^j)$ be a basis of $W^\star$, then $g= \sum_{i} b_i\otimes \alpha_1^i$ for some coefficients $b_i\in V_2$.
Let us call $\beta^j$ the image of $\alpha^j$ by $W^\star\to (V_1)^\star$, then $f= \sum_{i} b_i\otimes \beta_1^i$.
Reciprocally, if $f\in (V_1)^\star\otimes V_2$, it can be written as $f= \sum_{i\in I} b_i\otimes \gamma^i$ where $\gamma^i\in (V_1)^\star$, $b_i\in V_2$ and $I$ is finite. The $\gamma^i$ define a map $V_i\to k^I$ and $f$ factors through $k^I$.
\end{proof}

\begin{lemma}\label{intertens}
\begin{enumerate}
\item If $A'\subset A$ and $B'\subset B$ in $\cV$, then $A'\otimes B\cap A\otimes B' = A'\otimes B'$.
\item For $A_i\subset A$ and $B_i\subset B$ ($i=1,2$), we have $(A_1\otimes B_1)\cap (A_2\otimes B_2) = (A_1\cap A_2)\otimes (B_1\cap B_2)$.
\end{enumerate}
\end{lemma}
\begin{proof}
(1) We use decompositions $A=A'\oplus A''$ and $B=B'\oplus B''$.
We have $A\otimes B = A''\otimes B''\oplus A'\otimes B''\cap A''\otimes B' \oplus A'\otimes B'$.
The result follows from $A'\otimes B = A'\otimes B' \oplus A'\otimes B''$ and $A\otimes B' = A'\otimes B' \oplus A''\otimes B'$.

(2) We use (1) to get $A_i\otimes B_i = A_i\otimes B\cap A\otimes B_i$, then the result is a consequence of the left exactness of $\otimes$.
\end{proof}

\medskip
If $X$ is a set, and $V\in\cV^X$ we shall denote by $V^\star$ the object of $\cV^X$ which is the componentwise dual of $V$.
Recall that, if $X$ were a category, $V^\star$ would live in $\cV^{X^{op}}$, hence we are going to use the contravariant notation for components of $V^\star$ when $V$ has covariant components and reciprocally: $(A^\star)^x = (A_x)^\star$ and $(C^\star)_x = (C^x)^\star$.

Monomorphisms in $\cV^X$ are componentwise monomorphisms in $\cV$.
For $V\in \cV^X$, the canonical map $V\to V^{\star\star}$ is a monomorphism.
If $P$ is an $X$-colored operad and $A\in \cV^X$, there exists a canonical map $P^\wedge(A^\star)\to P(A)^\star$ given in components by
$$
[P^{\overline{x}}_y,(A^\star)^{\otimes \overline{x}}] \tto [P^{\overline{x}}_y,(A_{\otimes \overline{x}})^\star] = [P^{\overline{x}}_yA_{\otimes \overline{x}},k].
$$
This map is a monomorphism since ends preserve limits in their second variable.

\medskip
The following lemma is a useful property of the $P^\wedge_n$ when $\cV=\Vect$ or $\dgVect$.

\begin{lemma}\label{prodfibPwedge}
The functors $P^\wedge_n$ preserve intersections of subobjects (fiber products of monomorphisms).
In particular, they preserves monomorphisms.
\end{lemma}
\begin{proof}
By lemma \ref{intertens}, the functor $\epsilon:\cV^X\to \cV^{S(X)^{op}}$ sending $C$ to the family of $C^{\otimes \overline{x}}$ preserves 
intersections of subobjects and thus monomorphisms.
Then $P^\wedge_n$ has the same property since it is the composition of $\epsilon$ with an end.
\end{proof}
The lemma proves that the canonical map $P^\wedge(C) \to P^\wedge(C^{\star\star})$ is a monomorphism.
Composing with the previous map we get a monomorphism
$$
P^\wedge(C) \tto P^\wedge(C^{\star\star}) \tto P(C^\star)^\star.
$$
We can then faithfully think elements of $P^\wedge(C)$ as functions on $P(C^\star)$.

\subsection{Representative functions}\label{represfct}

For $V\in \cV^X$, we shall call an {\em element of $V$} any element of any of the components $V^x$ (or $V_x$) of $V$.

Let $P$ be an $X$-colored operad and $A\in \cV^X$, 
for any $y\in X$, we shall say that a linear map $f:P_y(A)\to k$, or equivalently an element of $(P(A)^\star)^y$, is a {\em representative function} if,
for any operation $\mu\in P^{\overline{x}}_y$,
there exists some functions $g^{(k,i)}(\mu,-)$ (depending linearly on $\mu$) such that, 
for any $a_i\in P(A)$ of adequate colors,
$$
f(\mu(a_1,\dots,a_n)) = \sum_i g^{(1,i)}(\mu,a_1)\otimes \dots \otimes g^{(n,i)}(\mu,a_n).
$$
We shall abbreviate such a sum using Sweedler's notation:
$$
f(\mu(a_1,\dots,a_n)) = g^{(1)}(\mu,a_1)\otimes \dots \otimes g^{(n)}(\mu,a_n).
$$
We shall also use the notation $\Delta^{(\mu)}f$ for the composition $f\mu$.

\begin{ex}
If $P=As$ is the associative operad, and $\mu$ is taken to be the binary multiplication operation, this definition implies that a representative function is representative in the sense of \cite{BL, Hazewinkel}. The two conditions are in fact equivalent. But, for an operad where no generators have been specified, we are forced to consider the stronger condition with all operations.			
\end{ex}

Let $P^{rep}(A)^y\subset (P(A)_y)^\star$ be the set of representative functions. Another way to understand the definition if to say that $P^{rep}(A)^y$ is defined as the fiber product
$$\xymatrix{
P^{rep}(A)^y\ar[rr]\ar[d]&& [P_y^{\overline{x}}, \otimes_i(P(A)_{x_i})^\star] = P^\wedge(P(A)^\star)^y \ar[d]\\
(P(A)_y)^\star\ar[rr]^-\Delta &&  [P_y^{\overline{x}}, (\otimes_iP(A)_{x_i})^\star] = (P^2(A)_y)^\star.
}$$
The component of $\Delta f$ in $[P_y^{\overline{x}}, (\otimes_iP(A)_{x_i})^\star]$ (no end taken here) evaluated at $\mu$ is $\Delta^{(\mu)}f= f\mu$.
$f$ is representative iff
$$
\Delta^{(\mu)}f = f\mu \in \bigotimes_i(P(A)_{x_i})^\star.
$$
We shall call $P^{2\star}$ the functor $A\mapsto P^2(A)^\star$.

\begin{lemma}\label{basecube}
There exists a commutative diagram
$$\xymatrix{
P^\wedge(C)\ar[rr]\ar[d]&& P(C^\star)^\star\ar[d] && (P^\wedge)^2(C)\ar[ll]\ar[d]\\
P^\wedge_2(C) \ar[rr] && P^{2\star}(C^\star) &&\ar[ll] P^\wedge(P(C^\star)^\star)
}$$
\end{lemma}
\begin{proof}
In components the left square is
$$\xymatrix{
[P^{\overline{x}}_y,C^{\otimes \overline{x}}]\ar[rr]\ar[d]&& [P^{\overline{x}}_y,((C^\star)_{\overline{x}})^\star] \ar[d]\\
[P^{\overline{z}}_y\PP^{\overline{x}}_{\overline{z}},C^{\otimes \overline{x}}] \ar[rr] && [P^{\overline{z}}_y\PP^{\overline{x}}_{\overline{z}},((C^\star)_{\overline{x}})^\star]
}$$
which is a simple consequence of the dinaturality of the end in its variables.
The reasoning is the same for the right square.
\end{proof}

\medskip
We saw in lemma \ref{Qmono} that the map $P^\wedge(P^\wedge(C))\to P^\wedge_2(C)$ is a monomorphism.
We shall say that an element of $P^\wedge(C)$ is {\em representative} if its image by $\Delta:P^\wedge(C)\to P^\wedge_2(C)$ belongs to $P^\wedge(P^\wedge(C))$, that is the subobject of representative elements is the object $Q_1(C)$ of section \ref{sectioncoreflection} defined as the fiber product
$$\xymatrix{
Q_1(C)\ar[r]\ar[d]&(P^\wedge)^2(C) \ar[d]&\\
P^\wedge(C)\ar[r]^{\Delta} & P^\wedge_2(C).
}$$

The following lemma is useful to work with representative elements as representative functions are more handy.
\begin{lemma}\label{represrepres}
An element of $P^\wedge(C)$ is representative iff it is a representative function in $P(C^\star)^\star$.
\end{lemma}
\begin{proof}

We have the following commutative cube
$$\xymatrix{
Q_1(C)\ar[rr]\ar[dd]\ar[rd]&& (P^\wedge)^2(C)\ar'[d][dd]\ar[rd]\\
&P^{rep}(C^\star)\ar[rr]\ar[dd]&& P^\wedge(P(C^\star)^\star)\ar[dd]\\
P^\wedge(C)\ar'[r][rr]\ar[rd]&& P^\wedge_2(C)\ar[rd]\\
&P(C^\star)^\star\ar[rr]&& P^{2\star}(C^\star).
}$$
The back face is cartesian by definition of $Q_1$.
The front face is cartesian by definition of $P^{rep}(C^\star)$.
The base and right faces are commutative by lemma \ref{basecube}.
Then, the left and top faces are deduced from the commutation of other faces and by the cartesian property of the front face.

The lemma will be proved if the left face is proven to be cartesian, this will be a consequence of the three others vertical faces being cartesian.
We need to prove that the right face is cartesian. 
Since the end with $P^{\overline{x}}_y$ preserves limits, it is enough to prove that the square 
$$\xymatrix{
\otimes_i[P_{x_i}^{\overline{z_i}}, C^{\otimes \overline{z_i}}]  \ar[d]\ar[rr]&& \otimes_i(P(C^\star)_{x_i})^\star\ar[d]\\
[\otimes_i P_{x_i}^{\overline{z_i}}, C^{\otimes \overline{z_1}\dots \overline{z_n}}]\ar[rr]&& (\otimes_iP(C^\star)_{x_i})^\star
}$$
is cartesian.
This follows essentially from the cartesian squares in $\cV$
$$\xymatrix{
[A,A']\otimes [B,B']\ar[rr]\ar[d] &&[A,A'']\otimes [B,B'']\ar[d]\\
[A\otimes B,A'\otimes B']\ar[rr] &&[A\otimes B,A''\otimes B'']
}$$
where $A'\subset A''$ and $B'\subset B''$. This can be proven using a decomposition $A''=A'\oplus A'''$ and $B''=B'\oplus B'''$.
\end{proof}

\subsection{Translations and recursive functions}\label{recursivefct}

We keep the previous notations.
Let $f$ be a function $P(A)_y\to k$, 
let us fix an operation $\mu\in P^{\overline{x}}_y$ together with a distinguished position $1\leq k\leq n$ (corresponding to the color $x_k$)
and let us fix elements $a_i\in P(A)_{x_i}$ for any $i\not=k$.
The {\em $\mu$-translation} of $f$ by the $(a_i)$ is the function $(\Delta^{(\mu)}f)^k_{(a_i)}:P(A)_{x_k}\to k$ defined by
$$
(\Delta^{(\mu)}f)^k_{(a_i)}(b) = f(\mu(a_1,\dots, b,\dots a_n))
$$
where $b$ is put in the $k$th position.
We shall say that a function $f:P(A)_y\to k$ is {\em recursive} if, for every operation $\mu\in P$, the space of all $\mu$-translations is of finite dimension. 
We emphasize that this finiteness condition is for every $\mu$ separately.
The space of $\mu$-translations is properly defined as the image of the map
\begin{eqnarray*}
\lambda^k(f\mu):P(A)_{x_1}\otimes \dots \otimes P(A)_{x_n} &\tto& (P(A)_{x_k})^\star\\
A_1,\dots a_n &\mto & b\mapsto f(\mu(a_1,\dots, b,\dots a_n))
\end{eqnarray*}
where $P(A_{x_k})$ is missing in the tensor product.
By lemma \ref{liftlemma}, $f$ is recursive iff 
$$
\Delta^{(\mu)}f = f\mu\in \bigcap_k\Big(\bigotimes_{i\not=k} P(A)_{x_i}\Big)^\star\otimes (P(A)_{x_k})^\star.
$$


The following proposition is the main tool for representative functions.
\begin{prop}\label{rep=rec}
A function $f$ is representative iff it is recursive.
\end{prop}
\begin{proof}
The proposition will be proven if we prove that $\bigcap_k(\otimes_{i\not=k} P(A)_{x_i})^\star\otimes (P(A)_{x_k})^\star = \otimes_i(P(A)_{x_i})^\star$.
The obvious inclusion proves that any $f$ representative is recursive.
The other inclusion is a consequence of the following lemma.
\end{proof}

\begin{lemma}
Let $f\in(\bigotimes_{i=1}^{n+1}V_i)^\star$ such that, for any $k\leq  n$, $f\in (V_k)^\star\otimes (\bigotimes_{i\not= k}V_i)^\star$, then 
$f\in \bigotimes_{i=1}^{n+1} (V_i)^\star$.
\end{lemma}
\begin{proof}
The result is obvious for $n=0,1$. We prove the rest by an induction incremented at $n=2$.
Let $f\in (V_1\otimes V_2\otimes V_3)^\star$
If $f\in (V_1)^\star \otimes (V_2\otimes V_3)^\star$, then for every $x_2\in V_2$
$f(-,x_2,-)\in (V_1)^\star \otimes (V_3)^\star$, hence $f$ is in $f\in [V_2, (V_1)^\star\otimes (V_3)^\star]$.
If moreover $f\in (V_2)^\star \otimes (V_1\otimes V_3)^\star$, we can take out the $V_2$ factor and this proves that $f\in (V_2)^\star\otimes (V_1)^\star\otimes (V_3)^\star$.

Now we assume the property is true for every sequence $V_i$ of length $n+1$.
Let $V_i$ be a sequence of length $n+2$. 
We put $V'_k= \bigoplus_{i\not=k}V_i$ and $V''_{k,\ell}=\bigoplus_{i\not=k,\ell}V_i$.
We have
$$
\bigcap_{k\leq n+1} (V_k)^\star\otimes (V'_k)^\star
= 
\bigcap_{2\leq k\leq n+1} (V_k)^\star\otimes (V'_k)^\star
\cap 
\bigcap_{k\leq n} (V_k)^\star\otimes (V'_k)^\star
$$
By hypothesis, with $V_{n+1}\otimes V_{n+2}$ playing the role of $V_{n+1}$, we have
$$
\bigcap_{k\leq n} (V_k)^\star\otimes (V'_k)^\star = \bigotimes_{k\leq n} (V_k)^\star \otimes (V_{n+1}\otimes V_{n+2})^\star.
$$
Similarly we have
$$
\bigcap_{2\leq k\leq n+1} (V_k)^\star\otimes (V'_k)^\star = \bigotimes_{k\leq n} (V_k)^\star \otimes (V_{1}\otimes V_{n+2})^\star.
$$
Intersecting those two terms we get
$$
\bigotimes_{2\leq k\leq n} (V_k)^\star \otimes \Big( (V_{n+1})^\star \otimes (V_{1}\otimes V_{n+2})^\star \cap (V_1)^\star \otimes (V_{n+1}\otimes V_{n+2})^\star\big).
$$
Using the computation for $n=2$, the last part is $(V_1)^\star \otimes (V_{n+1})^\star \otimes (V_{n+2})^\star$ and finally
$$
\bigcap_{k\leq n+1} (V_k)^\star\otimes (V'_k)^\star
= \bigotimes_{1\leq k\leq n+2} (V_k)^\star.
$$
\end{proof}

The translation of a translation is again a translation, hence any translation of a recursive function is recursive since subspaces of a finite dimensional space are finite dimensional. This proves the following lemma.
\begin{lemma}\label{transrepres}
Any translation of a representative function is representative.
\end{lemma}

The following proposition is the main result for representative functions.
\begin{prop}\label{liftrepres}
If $f$ is representative, then the $g^{(i)}$ in $\Delta^{(\mu)}f=g^{(1)}\otimes \dots \otimes g^{(n)}$ can be chosen to be representative functions.
\end{prop}
\begin{proof}

If $f$ is representative, it is recursive and the space of translations $(\Delta^{(\mu)}f)^k_{(a_i)}$ is finite dimensional.
Let us fix a basis of this space given by some $(n-1)$-uplets $\alpha^j=(\alpha_i^j)_{i\not=k}$ and put $g^j= (\Delta^{(\mu)}f)^k_{(\alpha_i^j)}$.
By lemma \ref{transrepres} the $g^j$ are representative functions and $f$ can be written $f=\sum g^j\otimes h^j$ for some 
$h^j\in \bigotimes_{\ell\not=k}(P(A)_{x_\ell})^\star$.
As this is true for every $k$, this proves that $\Delta^{(\mu)}f$ is in fact in the subspace of $\bigotimes_\ell(P(A)_{x_\ell})^\star$ given by
$$
\bigcap_k P^{rec}(A)_{x_k}\otimes \bigotimes_{\ell\not=k}(P(A)_{x_\ell})^\star
$$
Using lemma \ref{intertens} this intersection is simply $\bigotimes_{\ell} P^{rec}(A)_{x_\ell}$. This proves the result.
\end{proof}

The previous proposition proves that if $f:P(A)_y\to k$ is representative then
$\Delta f\in [P^{\overline{x}}_y, \bigotimes_{i} P^{rec}(A)_{x_i}] = P^\wedge(P^{rep}(A))$.
Equivalently, we have a commutative triangle
$$\xymatrix{
&P^\wedge(P^{rep}(C^\star))\ar[d]\\
P^{rep}(C^\star)\ar[r]\ar[ru]& P^\wedge(P(C^\star)^\star).
}$$

\subsection{Cofree coalgebras and representative functions}

Here is the main consequence of proposition \ref{liftrepres}.

\begin{thm}\label{mainthm2}
$Q_1(C)$ is a $P$-coalgebra, moreover this structure is natural in $C$.
\end{thm}
\begin{proof}
We consider the following diagram
$$
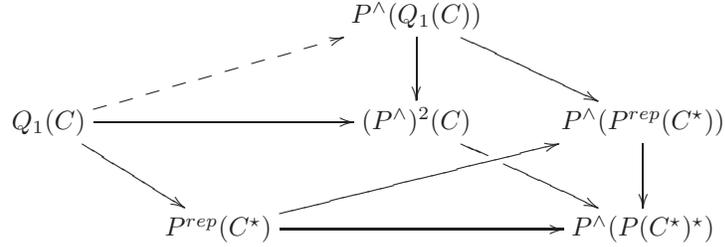
\xymatrix{
&&P^\wedge(Q_1(C))\ar[d]\ar[rd]\\
Q_1(C)\ar[rr]\ar[rd]\ar@{-->}[rru]&& (P^\wedge)^2(C)\ar[rd]|(.35)\hole &P^\wedge(P^{rep}(C^\star))\ar[d]\\
&P^{rep}(C^\star)\ar[rr]\ar[rru]&& P^\wedge(P(C^\star)^\star)\\
}$$

The existence of the front face is a reformulation of proposition \ref{liftrepres}.
The right face is the image by $P^\wedge$ of the left face of the cube of lemma \ref{represrepres}).
This face is a cartesian square of monomorphisms, by lemma \ref{prodfibPwedge}, the image by $P^\wedge$ is still a cartesian square of monomorphisms.
Finally, the dashed arrow is constructed using the cartesian structure of the right face.
The map $\delta:Q_1(C)\to P^\wedge(Q_1(C))$ will be the coalgebra structure on $Q_1(C)$.
$\delta$ is natural in $C$ since all maps of the diagram are.

To prove the coassociativity we consider the following diagram.
$$
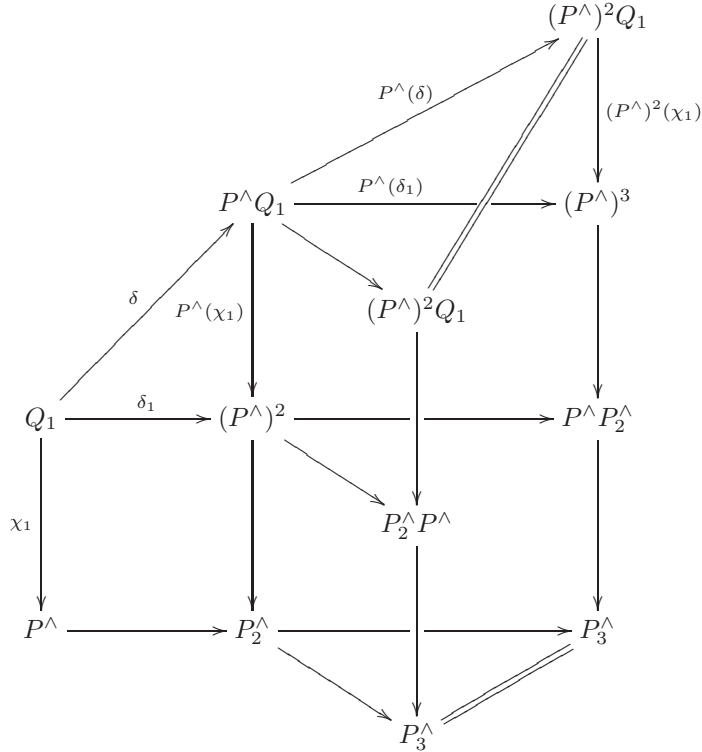
\xymatrix{
&&&&(P^\wedge)^2Q_1\ar[dd]^{(P^\wedge)^2(\chi_1)}\ar@{=}[lddd]\\
\\
&&P^\wedge Q_1 \ar[rruu]^{P^\wedge(\delta)}\ar[rr]|(.67)\hole^(.4){P^\wedge(\delta_1)} \ar[dd]_{P^\wedge(\chi_1)}\ar[dr]&& (P^\wedge)^3\ar[dd]\\
&&&(P^\wedge)^2Q_1\ar[dd]&\\
Q_1\ar[rr]^{\delta_1}\ar[rruu]^-{\delta} \ar[dd]_{\chi_1}&&(P^\wedge)^2\ar'[r][rr] \ar[dd]\ar[rd]&& P^\wedge P^\wedge_2\ar[dd]\\
&&& P^\wedge_2P^\wedge\ar[dd]&\\
P^\wedge \ar[rr]&& P^\wedge_2 \ar'[r][rr] \ar[rd]&& P^\wedge_3\ar@{=}[ld]\\
&&&P^\wedge_3
}$$
We have not indicated the name of all maps, hoping the missing ones should be clear enough, they are all constructed from $\chi_1$,  and the structure maps of the lax comonad $P^\wedge$.
The diagram is not fully commutative.
The back squares and triangles are commutative, 
the left and right sides of the prism are also commutative, but the top and bottom triangle are not.
Since everything is natural in $C$, we have withdrawn it from the notation.

The coassociativity is equivalent to the commutation of the top fork.
The diagram reduce this condition to the commutation of the bottom fork because vertical maps are monomorphisms.
The commutation of the bottom fork is a consequence of the simplicial diagram of $P^\wedge_n$.

The counit condition is left to the reader.
\end{proof}

The next result proves that the recursion of section \ref{sectioncoreflection} stops at the first step. 

\begin{cor}\label{cormainthm2}
$Q_1$ is the cofree $P$-coalgebra functor.
\end{cor}
\begin{proof}
With the notations of section \ref{sectioncoreflection}, we want to prove that $Q_1=Q_\infty$. 
It is actually enough to prove that $Q_2=Q_1$ since the whole tower will be constant under this hypothesis.

Since $Q_1(C)$ is a $P$-coalgebra, it is a $Q_\infty$-coalgebra, hence there exists a map $\delta:Q_1(C)\to Q_\infty Q_1(C)$ lifting $Q_1(C)\to P^\wedge Q_1(C)$ and satisfying a coassociativity condition.
Using the projections $Q_\infty\to Q_1$ and $Q_1\to P^\wedge$ we get from this condition a commutative diagram
$$\xymatrix{
&& (Q_1)^2(C) \ar@<.6ex>[r]\ar@<-.6ex>[r] \ar[d]&(Q_1)^3(C)\ar[r]&(P^\wedge)^3(C)\\
Q_1(C)\ar[rr]\ar[rru]^{\delta}&& (P^\wedge)^2(C) .
}$$
By definition of $Q_2$, we get a map $Q_1(C)\to Q_2(C)$ but since $Q_2(C)\to Q_1(C)$ is a monomorphism, this proves that $Q_1=Q_2$ and the result.
\end{proof}

In other terms, the cofree $P$-coalgebra $P^\vee(C)$ on $C$ is the subobject of $P^\wedge$ defined by the fiber product
$$\xymatrix{
P^\vee(C)\ar[rr]\ar[d]&& (P^\wedge)^2(C)\ar[d]\\
P^\wedge(C)\ar[rr]&& P^\wedge_2(C).
}$$

\begin{rem}\label{remsmithfinal}
This proves also that the recursion of \cite{Smith} stops at the first step since our $Q_1$ is his $L_1$ (see remark \ref{remsmith}).
\end{rem}

\end{document}